\documentclass[11pt, a4paper]{article}

\usepackage[utf8]{inputenc}
\usepackage[english]{babel}
\usepackage[T1]{fontenc}
\usepackage[colorlinks = true,
            linkcolor = blue,
            urlcolor  = blue,
            citecolor = blue,
            anchorcolor = blue]{hyperref}
\usepackage[a4paper,width=150mm,top=25mm,bottom=25mm]{geometry}
\usepackage{fancyhdr}
\usepackage[symbol]{footmisc}

\usepackage{amsmath,amsthm,amssymb, enumitem, mathtools, mathrsfs}
\usepackage{siunitx}
\setitemize{noitemsep}

\usepackage{tikz}
\usepackage{tikz-cd}
\usepackage{caption}
\usepackage[all]{xy}
\usepackage{graphicx}
\usepackage{chngcntr}
\usepackage{booktabs}
\usepackage[labelformat=simple]{subcaption}

\usepackage{csquotes}
\usepackage[backend=biber,
            sorting=anyt]{biblatex}
\graphicspath{ {images/} }

\theoremstyle{plain}
\newtheorem{theorem}{Theorem}[section]
\newtheorem*{theorem*}{Theorem}
\newtheorem{proposition}[theorem]{Proposition}
\newtheorem{lemma}[theorem]{Lemma}

\theoremstyle{definition}
\newtheorem{definition}[theorem]{Definition}
\newtheorem{assumption}[theorem]{Assumption}
\newtheorem{example}[theorem]{Example}

\theoremstyle{remark}
\newtheorem{remark}[theorem]{Remark}

\numberwithin{equation}{section}

\DeclareMathOperator{\sgn}{sgn}

\newcommand{\ev}{\mathbb{E}}
\newcommand{\pr}{\mathbb{P}}
\newcommand{\qr}{\mathbb{Q}}
\newcommand{\R}{\mathbb{R}}
\renewcommand{\P}{\mathcal{P}}
\newcommand{\F}{\mathcal{F}}
\renewcommand{\L}{\mathcal{L}}

\newcommand{\Leb}{\mathrm{Leb}}
\renewcommand{\d}{\mathrm{d}}
\newcommand{\define}{\mathpunct{:}}

\newcommand{\bb}[1]{\mathbb{#1}}

\renewcommand{\cal}[1]{\mathcal{#1}}

\begin{document}
\noindent
\begin{center}
    \Large
    \textbf{Control of Conditional Processes \\
    and Fleming--Viot Dynamics}

    \vspace{1em}

    \normalsize
    Philipp Jettkant\footnote[1]{Department of Mathematics, Imperial College London, United Kingdom, \href{mailto:p.jettkant@imperial.ac.uk}{p.jettkant@imperial.ac.uk}.}

\end{center}

\vspace{1em}

\begin{abstract}
We discuss equivalent formulations of the control of conditional processes introduced by Lions. In this problem, a controlled diffusion process is killed once it hits the boundary of a given domain and the controller's reward is computed based on the conditional distribution given the process's survival. So far there is no clarity regarding the relationship between the open- and closed-loop formulation of this nonstandard control problem. We provide a short proof of their equivalence using measurable selection and mimicking arguments. In addition, we link the closed-loop formulation to Fleming--Viot dynamics of McKean--Vlasov type, where upon being killed the diffusion process is reinserted into the domain according to the current law of the process itself. This connection offers a new interpretation of the control problem and opens it up to applications that feature costs caused by reinsertion.
\end{abstract}

\section{Introduction} \label{sec:introduction}

In this article, we consider the control of conditional processes, originally introduced by Lions in a series of lectures at the Coll\`ege de France \cite{lions_cond_proc_2016}. Fix a bounded and open subset $D$ of $\R^d$ and let $(\Omega, \F, \pr)$ be a probability space carrying a filtration $\bb{F} = (\F_t)_{t \geq 0}$ with $\F_t \subset \F$ for $t \geq 0$, a $D$-valued $\F_0$-measurable random variable $\xi$, and a $d$-dimensional $\bb{F}$-Brownian motion $W = (W_t)_{t \geq 0}$. The $\R^d$-valued state process $X = (X_t)_{t \geq 0}$ follows the dynamics
\begin{equation} \label{eq:conditional_mv}
    \d X_t = b(t, X_t, \mu_t, \alpha_t) \, \d t + \sigma \, \d W_t
\end{equation}
with initial condition $X_0 = \xi$. Here $\sigma \in \R^{d \times d}$ is invertible and $\alpha = (\alpha_t)_{t \geq 0}$ is an $\bb{F}$-progressively measurable control process that takes values in some closed subset $A$ of $\R^{d_A}$. The probability measure $\mu_t = \L(X_t \vert \tau > t)$ is the conditional distribution of $X_t$ given that the diffusion has not exited the domain $D$ at time $t$, i.e.\@ $\tau = \inf\{s > 0 \define X_s \notin D\}$. We set $\mu = (\mu_t)_{t \geq 0}$, so that $\mu$ is an element of $C([0, \infty); \P(D))$, where the set $\P(D)$ of probability measures on $D$ is endowed with the topology of weak convergence. If we interpret exiting the domain as killing the diffusion, then $\mu_t$ is the conditional distribution of the surviving process. Since the coefficients appearing in SDE \eqref{eq:conditional_mv} depend on the conditional distribution of the solution $X$, we refer to SDE \eqref{eq:conditional_mv} as a conditional McKean--Vlasov SDE. This should be distinguished from McKean--Vlasov SDEs with a common noise, where one conditions the distribution of the solution on the common noise. 

The controller's objective is to maximise the reward
\begin{equation} \label{eq:cost_functional}
    J(\alpha, \mu) = \int_0^T\ev\bigl[f(t, X_t, \mu_t, \alpha_t) \big\vert \tau > t\bigr]  \, \d t + g(\mu_T)
\end{equation}
over a set $\cal{A}$ of \textit{admissible controls} $(\alpha, \mu)$, where $T > 0$ is some fixed time horizon. The set $\cal{A}$ consists of tuples $(\alpha, \mu)$ such that $\alpha$ is an $\bb{F}$-progressively measurable $A$-valued process and $\mu$ is an element of $C([0, \infty); \P(D))$, for which there exists a strong solution $X$ to the conditional McKean--Vlasov SDE \eqref{eq:conditional_mv} with $\mu_t = \L(X_t \vert \tau > t)$ for $t \geq 0$. We will refer to the elements of $\cal{A}$ also as \textit{open-loop controls}. We define the \textit{value} $V$ of the control problem by $V = \sup_{(\alpha, \mu) \in \cal{A}} J(\alpha, \mu)$. Note that the reason for making $\mu$ part of the control is that it is not clear whether the conditional McKean--Vlasov SDE \eqref{eq:conditional_mv} has a unique solution for any $\bb{F}$-progressively measurable $A$-valued process $\alpha$. We refer to an admissible control $(\alpha, \mu) \in \cal{A}$ as a \textit{closed-loop control} if $\alpha_t = a(t, X_t)$ for a measurable function $a \define [0, \infty) \times \R^d \to A$, where $X$ denotes the solution to the conditional McKean--Vlasov SDE \eqref{eq:conditional_mv} corresponding to $(\alpha, \mu)$. The map $a$ is called the \textit{feedback function} (associated with $(\alpha, \mu)$). The \textit{value} of the closed-loop formulation, i.e.\@ the supremum of $J$ over all closed-loop controls, is denoted by $V_{\textup{closed}}$. The closed-loop formulation of the control of conditional processes, its (deterministic) reformulation through the Fokker--Planck equation satisfied by $\mu_t$, and the asymptotic behaviour as $T \to \infty$ were studied by Achdou, Lauri\`ere \& Lions \cite{achdou_cond_oc_2021}. Nutz \& Zhang \cite{nutz_cond_stopping_2020} considered an optimal stopping problem in the conditional framework. Note that since the set of closed-loop controls is a subset of $\cal{A}$, it holds that $V_{\textup{closed}} \leq V$. In the first part of this paper, we shall be interested in establishing the reverse inequality. 

\subsection{Open- Versus Closed-Loop Formulation}

As it stands, there is a lack of clarity regarding the equivalence between the open- and closed-loop formulation of the control problem, that is, whether $V = V_{\textup{closed}}$. In his lectures at the Coll\`ege de France, Lions demonstrated that the control of conditional processes differs from standard control problems in that an optimal closed-loop control does not only depend on the current state $X_t$ but also the conditional distribution $\mu_t = \L(X_t \vert \tau > t)$\footnote[1]{Cf.\@ Lecture \href{https://www.college-de-france.fr/fr/agenda/cours/equations-de-hjb-et-extensions-de-la-theorie-classique-du-controle-stochastique/equations-de-hjb-et-extensions-de-la-theorie-classique-du-controle-stochastique-9}{9} of the course by Lions \cite{lions_cond_proc_2016}.}. This resembles the situation for the control of McKean--Vlasov SDEs, where optimal closed-loop controls vary with the law of the solution to the McKean--Vlasov SDE. In both cases, this additional dependence can be subsumed in the time variable of the feedback function $a$ associated with an optimal closed-loop control, suggesting that the supremum of $J$ over all open- and closed-loop controls may coincide, i.e.\@ $V = V_{\textup{closed}}$.
As is now well-known, the latter is not only the case for standard control problems, but also typically holds for  McKean--Vlasov control problems \cite{lacker_limit_theory_mve_2017, lacker_mimicking_2023}. Thus, it is natural to suspect the same for the control of conditional processes.

The first objective of the present article is to confirm this hypothesis. To this end, we adapt the strategy by Lacker \cite{lacker_limit_theory_mve_2017}, originally used to establish the equivalence between open- and closed-loop controls for standard McKean--Vlasov control problems, to conditional processes. This idea was suggested in a recent preprint \cite[Appendix A.4]{hambly_bspde_2024} by the author and B.\@ Hambly, but requires further theoretical justification, which we provide here while simplifying the arguments. A similar approach, more closely resembling the original proposal in \cite[Appendix A.4]{hambly_bspde_2024}, is pursued in a concurrent work by Carmona \& Lacker \cite{carmona_mimicking_conditional_2024}. The results in this paper were developed independently of \cite{carmona_mimicking_conditional_2024} and we provide a brief comparison in Remark \ref{rem:comp_car} in Section \ref{sec:equivalence} below. Let us outline the strategy from \cite{lacker_limit_theory_mve_2017} and explain our adaptation: starting from any open-loop control $(\alpha, \mu) \in \cal{A}$ with corresponding solution $X$ to SDE \eqref{eq:conditional_mv}, one employs a measurable selection theorem by Haussmann \& Lepeltier \cite[Theorem A.9]{haussmann_oc_1990} together with the mimicking theorem by Brunick \& Shreve \cite{brunick_mimicking_2013} to construct a feedback function $a$, derived from the control $\alpha$, as well as a solution $\tilde{X}$ to the SDE 
\begin{equation} \label{eq:mimicking_intro}
    \d \tilde{X}_t = b\bigl(t, \tilde{X}_t, \mu_t, a(t, \tilde{X}_t)\bigr) \, \d t + \sigma \, \d W_t,
\end{equation}
such that $\L(\tilde{X}_t) = \L(X_t)$. If SDE \eqref{eq:conditional_mv} were a classical McKean--Vlasov SDE, meaning that $\mu_t = \L(X_t) = \L(\tilde{X}_t)$, then $\tilde{X}$ would be a solution to this McKean--Vlasov SDE with control process $\tilde{\alpha}_t = a(t, \tilde{X}_t)$ and, under suitable convexity assumptions, the selection of $a$ through \cite[Theorem A.9]{haussmann_oc_1990} would guarantee that $J(\tilde{\alpha}, \mu) \geq J(\alpha, \mu)$. However, in order for $\tilde{X}$ to solve the \textit{conditional} McKean--Vlasov SDE \eqref{eq:conditional_mv}, we must have $\L(\tilde{X}_t \vert \tilde{\tau} > t) = \L(X_t \vert \tau > t) = \mu_t$, with $\tilde{\tau} = \inf\{t > 0 \define \tilde{X}_t \notin D\}$, instead of $\L(\tilde{X}_t) = \L(X_t) = \mu_t$. To achieve this, the key insight is that, under mild regularity conditions on the domain $D$, we can represent the conditional McKean--Vlasov SDE \eqref{eq:conditional_mv} equivalently through the system
\begin{equation} \label{eq:joint_mimicking}
    \d X_t = b(t, X_t, \mu_t, \alpha_t) \, \d t + \sigma \, \d W_t, \quad \d \Lambda_t = \mathbf{1}_{X_t \notin D} \, \d t
\end{equation}
with initial condition $X_0 = \xi$, $\Lambda_0 = 0$, and $\mu_t = \L(X_t \vert \Lambda_t = 0)$, using that a.s.\@ $\inf\{t > 0 \define X_t \notin D\} = \inf\{t > 0 \define \Lambda_t > 0\}$. We may then apply the mimicking theorem to the joint process $(X, \Lambda)$ to conclude.

It should be stressed that a closely related result was obtained by Campi, Ghio \& Livieri \cite{campi_mfg_hitting_2021}. They study mean-field games where players interact through their subprobability distribution $\nu_t = \pr(X_t \in \cdot,\, \tau > t)$ rather than the conditional distribution $\mu_t = \L(X_t \vert \tau > t)$. Likewise, the computation of the reward functional is based on the subprobability distribution, replacing $\ev[f(t, X_t, \mu_t, \alpha_t) \vert \tau > t]$ and $g(\mu_T)$ in \eqref{eq:cost_functional} with $\ev[\mathbf{1}_{\tau > t} f(t, X_t, \nu_t, \alpha_t)]$ and $g(\nu_T)$. Using ideas that resemble the approach we take here, they construct a closed-loop Nash equilibrium in \cite[Proposition 3.6]{campi_mfg_hitting_2021}. Their proof omits certain details associated with the removal of players at the exit time $\tau$. Instead of relying on the pair $(X, \Lambda)$ introduced in \eqref{eq:joint_mimicking}, they mimic the process $Y_t = (t \land \tau, X_{t \land \tau})$. However, it is not immediately evident whether the mimicking process $\tilde{Y} = (\tilde{Y}^1, \tilde{Y}^2)$ takes the intended form $(t \land \tilde{\tau}, \tilde{X}_{t \land \tilde{\tau}})$, where $\tilde{X}$ solves SDE \eqref{eq:mimicking_intro} and $\tilde{\tau} = \inf\{t > 0 \define \tilde{X}_t \notin D\}$. This issue can be resolved by applying the extended mimicking theorem by Brunick \& Shreve \cite[Theorem 3.6]{brunick_mimicking_2013} with an appropriate choice of \textit{updating function} (a concept introduced in \cite[Definition 3.1]{brunick_mimicking_2013}), but was not discussed in \cite{campi_mfg_hitting_2021}. Private communication with the article's authors revealed that they had considered the use of a suitable updating function, but ultimately this was not included in the article. Our approach circumvents the usage of updating functions altogether.

Let us also mention that in an earlier paper, Campi \& Fischer \cite{campi_mfg_absorption_2018} considered mean-field games where players interact through their conditional distribution $\L(X_t \vert \tau > t)$, while the computation of the reward $J$ is based on the subprobability distribution $\pr(X_t \in \cdot,\, \tau > t)$. They also show the existence of a closed-loop Nash equilibrium, but following an entirely different line of argument. They characterise Nash equilibria through a forward-backward SDE and, subsequently, express the solution of the backward SDE, and thus the optimal control of the representative player, as a function of the state process $X_t$ (and $t$). However, these arguments are specific to the mean-field game setting and cannot be transferred to the control of conditional processes. 

In a recent work, aimed at making progress towards confirming the equivalence of the open- and closed-loop formulation for the control of conditional processes, Carmona, Lauri\`ere \& Lions \cite{carmona_nssc_2023} studied a relaxation with ``soft'' as opposed to ``hard'' killing. Rather than killing the diffusion once the state $X_t$ hits the boundary of $D$, the process $X$ can exit the domain and is killed at a positive rate $\lambda(X_t) > 0$ whenever it is located outside of $\bar{D}$. Specifically, the diffusion is killed once $\Lambda_t = \int_0^t \lambda(X_s) \, \d s$ exceeds a standard exponential random variable. Accordingly, $\mu_t$ is replaced by $\ev[e^{-\Lambda_t} \delta_{X_t}]/\ev[e^{-\Lambda_t}]$, where the fraction $e^{-\Lambda_t}$ represents the probability that the diffusion is still alive. An analogous replacement is made in the conditioning appearing in the reward functional $J$ in \eqref{eq:cost_functional}. Note that for the specific choice $\lambda(x) = \infty \mathbf{1}_{\{x \notin \bar{D}\}}$, not covered in \cite{carmona_nssc_2023}, one recovers Lions' original problem, since then $e^{-\Lambda_t} = \mathbf{1}_{\{X_s \in \bar{D},\, 0 \leq s \leq t\}} = \mathbf{1}_{\{\tau > t\}}$ outside the nullset $\{\tau = t\}$. In the case of soft killing, the main challenge in establishing the equivalence between the open- and closed-loop formulation lies in proving that one does not need to track $\Lambda_t$ in order to optimally control the system. In other words, it suffices to consider feedback functions $a$ that depend on $t$ and the state $X_t$ but not the cumulative intensity $\Lambda_t$. This fact is closely related to the memoryless property of the exponential distribution. Carmona, Lauri\`ere \& Lions achieve this by studying two nonlocal PDEs of Hamilton--Jacobi--Bellman type that characterise optimal closed-loop controls with and without dependence on $\Lambda_t$. They show that both PDEs have the same unique solution, thereby guaranteeing that both types of closed-loop controls yield the same value. A similar equivalence result was obtained by Hambly \& Jettkant \cite{hambly_sfpe_2024} in the context of McKean--Vlasov control with soft killing and common noise. Its proof also builds on the ideas by Lacker \cite{lacker_limit_theory_mve_2017}, drawing on a generalisation of Gy\"ongy's mimicking theorem \cite{gyongy_mimicking_1986} to McKean--Vlasov SDEs with common noise by Lacker, Shkolnikov \& Zhang \cite{lacker_mimicking_2023}, thus avoiding the study of nonlocal PDEs. A key contribution of the present paper is to address the equivalence hypothesis directly for Lions' original model with hard killing, without introducing a relaxation of the control problem.

\subsection{Fleming--Viot Dynamics of McKean--Vlasov Type}

The second contribution of this article is to link the conditional McKean--Vlasov SDE \eqref{eq:conditional_mv} to Fleming--Viot dynamics of McKean--Vlasov type. This connection is inspired by the relationship between \eqref{eq:conditional_mv} and a Fleming--Viot type particle system established by Burdzy et al.\@ \cite{burdzy_flem_viot_1996, burdzy_flem_viot_2000}. The latter system consists of $N \geq 1$ Brownian particles $X^i_t$ in a domain $D$, which upon hitting the boundary $\partial D$ are reinserted at a position that is chosen uniformly at random amongst the current locations of the remaining particles. Burdzy, Ho\l{}yst \& March \cite{burdzy_flem_viot_2000} show that as $N \to \infty$, for each $t \geq 0$, the empirical measure $\frac{1}{N} \sum_{i = 1}^N \delta_{X^i_t}$ of the particles converges to the marginal distribution of a Brownian motion killed upon hitting $\partial D$, conditioned on survival. This dynamic is a special case of the conditional McKean--Vlasov SDE \eqref{eq:conditional_mv}, where $b$ vanishes and $\sigma$ is the identity matrix. A uniform-in-time version of this result was later proven by Grigorescu and Kang \cite{grigorescu_flem_viot_2004}. Subsequently, Tough and Nolen \cite{tough_fleming_viot_2022} generalised the theory to systems, where each particle's dynamics can feature a drift that may depend on the particle's location as well as the empirical measure, similar to \eqref{eq:conditional_mv}. In each of these cases, only the single marginal empirical distributions $\frac{1}{N} \sum_{i = 1}^N \delta_{X^i_t}$ of the particle system converge to the time marginals $\mu_t = \L(X_t \vert \tau > t)$ of the conditional process, while the trajectories of the two systems are quite different.

In this work, we introduce a mean-field analogue of the Fleming--Viot type particle system, that also captures the particle system's trajectorial behaviour. We refer to these mean-field dynamics as \textit{Fleming--Viot dynamics of McKean--Vlasov type}. The state $Y$ of the representative particle of this system follows the McKean--Vlasov SDE
\begin{equation} \label{eq:fleming_viot_intro}
    \d Y_t = b\bigl(t, Y_t, \L(Y_t), a(t, Y_t)\bigr) \, \d t + \sigma \, \d W_t
\end{equation}
for a feedback function $a \define [0, \infty) \times \R^d \to A$. In addition, if the particle hits the boundary $\partial D$ at some time $t \geq 0$, it is reinserted according to the prevailing law $\L(Y_t)$. This process is repeated whenever the particle exits $D$. The expected number of reinsertions up to time $t$ is denoted by $F_t$. We show that the above dynamics are well-posed, that the marginal law $\L(Y_t)$ coincides with the conditional distribution $\mu_t = \L(X_t \vert \tau > t)$ of the solution to the conditional McKean--Vlasov SDE \eqref{eq:conditional_mv} with closed-loop control $\alpha_t = a(t, X_t)$, and that $F_t = -\log \pr(\tau > t)$. Consequently, controlling the conditional McKean--Vlasov SDE \eqref{eq:conditional_mv} over feedback functions $a \define [0, \infty) \times \R^d \to A$ is equivalent to controlling the Fleming--Viot dynamics of McKean--Vlasov type \eqref{eq:fleming_viot_intro} over such $a$. 

The dynamics \eqref{eq:fleming_viot_intro} open up the control problem for conditional processes to a larger class of applications, where reinsertion is penalised by adding the cost $-c F_T$, for $c > 0$, to the reward functional. It is worth noting that this term is already incorporated in the work by Achdou, Lauri\`ere \& Lions \cite[Equation (6)]{achdou_cond_oc_2021} to facilitate the passage to the limit as $T \to \infty$, though it is not interpreted as a reinsertion cost. A concrete example of a model with reinsertion is the following: $Y_t$ represents the workload of a machine of a large manufacturing company. The workload is managed by the company's employees through the control $a(t, Y_t)$, based on the current workload, but is also subject to random shocks, modelled by a Brownian motion, so that $\d Y_t = a(t, Y_t) \, \d t + \sigma \, \d W_t$. Increasing the machine's workload raises the company's revenues, but also makes the machine more susceptible to defects. If the workload becomes too high, the machine breaks down and has to be replaced at the cost $c$. The company's objective is to find the right balance between revenue generation and minimisation of operating costs. We leave the analysis of the control problem for McKean--Vlasov SDE \eqref{eq:fleming_viot_intro} in the presence of reinsertion costs and its applications for future work.

\section{Equivalence between the Open- and Closed-Loop Formulation} \label{sec:equivalence}

In this section, we establish the equivalence between the open- and closed-loop formulation of the control of conditional processes, see Theorem \ref{thm:closed_from_weak}. We begin by stating the assumptions under which we will be working in the remainder of this section. Since we restrict our attention to the compact interval $[0, T]$ throughout this article, we assume for simplicity that the coefficient $b$ vanishes for $t \in (T, \infty)$. Consequently, if we apply a change of measure \`a la Girsanov to remove or adjust the drift, it suffices to evaluate the stochastic exponential, which provides the density of the new measure, at time $T$.

\begin{assumption} \label{ass:control_problem}
We assume that $b \define [0, \infty) \times \R^d \times \P(D) \times A \to \R^d$, $f \define [0, \infty) \times \R^d \times \P(D) \times A \to \R$, and $g \define \P(D) \to \R$ are measurable and bounded functions, $\sigma \in \R^{d \times d}$, and that there exists a constant $C > 0$ such that
\begin{enumerate}
    \item \label{it:coefficient} for all $(t, x, a) \in [0, T] \times \R^d \times A$ and $m$, $m' \in \P(D)$, we have
    \begin{equation*}
        \bigl\lvert b(t, x, m, a) - b(t, x, m', a)\rvert \leq C\lVert m - m'\rVert_{\textup{TV}};
    \end{equation*}
    \item \label{it:nondeg} $\sigma$ is invertible.
\end{enumerate}
Here $\lVert \cdot \rVert_{\textup{TV}}$ denotes the total variation norm on the space of finite signed measures on $D$.
\end{assumption}

The above condition on the drift $b$ differs from the typically assumed Lipschitz continuity in space and the measure argument (the latter being with respect to some Wasserstein metric on the space of probability measures). This reflects the fact that solutions to the conditional McKean--Vlasov SDE \eqref{eq:conditional_mv} cannot be constructed through a standard contraction argument on a suitable Banach space of stochastic processes with continuous trajectories. Instead, we rely on total variation estimates for SDEs presented by Campi \& Fischer \cite[Proposition C.1]{campi_mfg_absorption_2018}. These allow us to obtain the flow $\mu = (\mu_t)_{0 \leq t \leq T}$ of conditional probabilities in \eqref{eq:conditional_mv} as the fixed point of a contraction on the space $C([0, T]; \P(\R^d))$ in the total variation distance.

The following (semi)continuity and convexity assumption is required to apply a measurable selection theorem from Haussmann \& Lepeltier \cite[Theorem A.9]{haussmann_oc_1990}. It allows us to construct a closed-loop from an open-loop control.

\begin{assumption} \label{ass:convex}
We assume that for any $(t, x, m) \in [0, T] \times \R^d \times D$, the functions $A \ni a \mapsto b(t, x, m, a)$ and $A \ni a \mapsto f(t, x, m, a)$ are continuous and upper semicontinuous, respectively, and the set
\begin{equation*}
    \Bigl\{\bigl(b(t, x, m, a), z\bigr) \define z \in \R,\, a \in A,\, z \leq f(t, x, m, a)\Bigr\} \subset \R^d \times \R
\end{equation*}
is closed and convex.
\end{assumption}

Our goal is to show that $V_{\textup{closed}} = V$. We achieve this by obtaining for any given admissible control $(\alpha, \mu)$ a closed-loop control that achieves the same or a greater reward. In the remainder of this section, we will denote the solution to the conditional McKean--Vlasov SDE \eqref{eq:conditional_mv} associated to an admissible control $(\alpha, \mu)$ by $X^{\alpha, \mu}$. We write $\tau_{\alpha, \mu}$ for the corresponding hitting time of $D^c$. As mentioned in Section \ref{sec:introduction}, the uniqueness of the conditional McKean--Vlasov SDE \eqref{eq:conditional_mv} for an arbitrary $\bb{F}$-progressively measurable $A$-valued process $\alpha$ is not obvious, due to the comparably irregular behaviour of the hitting time $\tau_{\alpha, \mu}$. This is why our definition of an admissible control includes the flow $\mu$ of the conditional law of the solution. However, as we show next, starting from a feedback function $a \define [0, \infty) \times \R^d \to A$, we can construct a unique strong solution to the conditional McKean--Vlasov SDE
\begin{equation} \label{eq:conditional_mv_closed}
    \d X_t = b\bigl(t, X_t, \mu_t, a(t, X_t)\bigr) \, \d t + \sigma \, \d W_t
\end{equation}
with $\mu_t = \L(X_t \vert \tau > t)$ and $\tau = \inf\{t > 0 \define X_t \notin D\}$. Then, the pair $(\alpha, \mu)$, given by $\alpha_t = a(t, X_t)$ and $\mu = (\mu_t)_{0 \leq t \leq T}$, yields a closed-loop control. In particular, the set of admissible controls is not empty. Since the feedback function $a$ completely determines the associated closed-loop control $(\alpha, \mu)$, we will sometimes refer to $a$ itself as a closed-loop control. Moreover, we define the closed-loop reward functional $J_{\textup{CP}}$ by $J_{\textup{CP}}(a) = J(\alpha, \mu)$.

Before we state the existence and uniqueness result for the conditional McKean--Vlasov SDE \eqref{eq:conditional_mv_closed}, we need to make a couple of definitions. We equip the space $C([0, \infty); \R^d)$ with the metric
\begin{equation*}
    (x, y) \mapsto \sum_{k = 1}^{\infty} 2^{-k} \bigl(\lVert x_{\cdot \land k} - y_{\cdot \land k} \rVert_{\infty} \land 1 \bigr).
\end{equation*}
The total variation norm on the space of finite signed measures on $C([0, \infty); \R^d)$ with respect to this metric is denoted by $\lVert \cdot \rVert_{\textup{TV}}$. Next, for a measure $\nu \in \P(C([0, \infty); \R^d))$ and $t \geq 0$, we denote by $\nu_t \in \P(C([0, \infty); \R^d))$ the pushforward of $\nu$ under the map $C([0, \infty); \R^d) \to C([0, \infty); \R^d)$, $x \mapsto x_{\cdot \land t}$. Finally, we let $C([0, T]; \P(\R^d))$ be the space of flows of measures on $\P(\R^d)$ that are continuous with respect to the bounded Lipschitz distance on $\P(\R^d)$. Note that while continuity of the flows is measured with respect to the bounded Lipschitz distance, in the proof of Proposition \ref{prop:exist_unique} below, the space $C([0, T]; \P(\R^d))$ itself will be endowed with a metric involving the total variation norm on $\P(\R^d)$.

\begin{proposition} \label{prop:exist_unique}
Let Assumption \ref{ass:control_problem} be satisfied. Then for any measurable function $a \define [0, \infty) \times \R^d \to A$, the conditional McKean--Vlasov SDE \eqref{eq:conditional_mv_closed} has a unique strong solution. In particular, the set $\cal{A}$ of admissible controls is nonempty.
\end{proposition}

\begin{proof}
We introduce a map $\Phi \define C([0, T]; \P(\R^d)) \to C([0, T]; \P(\R^d))$ defined as follows: for $\mu \in C([0, T]; \P(\R^d))$, let $\nu^{\mu} \in \P(C([0, \infty); \R^d))$ denote the distribution of the unique strong solution of the SDE
\begin{equation} \label{eq:sde_mu_fixed}
    \d X^{\mu}_t = b\bigl(t, X^{\mu}_t, \mu_{t \land T}, a(t, X^{\mu}_t)\bigr) \, \d t + \sigma \, \d W_t
\end{equation}
with initial condition $X^{\mu}_0 = \xi$ and set $\tau_{\mu} = \inf\{t > 0 \define X^{\mu}_t \notin D\}$. The existence of a unique strong solution to SDE \eqref{eq:sde_mu_fixed} is guaranteed by a well-known result of Veretennikov \cite[Theorem 1]{veretennikov_strong_ex_1981}. Note that since $b$ vanishes after $T$, stopping $\mu$ at $T$ does not pose any issues. Then, $\Phi(\mu)$ is given by $\Phi(\mu)_t = \L(X^{\mu}_t \vert \tau_{\mu} > t)$ for $t \in [0, T]$. Note that $[0, T] \ni t \mapsto \L(X^{\mu}_t \vert \tau_{\mu} > t)$ is continuous with respect to the bounded Lipschitz distance on $\P(\R^d)$, so $\Phi$ is well-defined. The former statement uses the fact that $\pr(\tau_{\mu} = t) = 0$ for $t \geq 0$. Since we could not find a reference for this result in the literature, we report a proof in the appendix, see Proposition \ref{prop:hitting_no_atoms}.

We will show that the map $\Phi$ is a contraction on $C([0, T]; \P(\R^d))$ with respect to an appropriately chosen metric. We follow the arguments from the proof of \cite[Proposition C.1]{campi_mfg_absorption_2018}, see also \cite[Section 8.3]{tough_fleming_viot_2022}, designed to show uniqueness of equations of the same type as McKean--Vlasov SDE \eqref{eq:conditional_mv_closed}. Let $\mu^1$, $\mu^2 \in C([0, T]; \P(\R^d))$, set $\nu^i = \nu^{\mu_i}$, $i = 1$, $2$, and let $X = X^{\mu^1}$.
Next, we introduce the measure $\qr$ defined by
\begin{equation*}
    \frac{\d \qr}{\d \pr}\bigg\rvert_{\F} = \cal{E}\biggl(\int_0^{\cdot} \sigma^{-1} \bigl(b_2(t, X_t) - b_1(t, X_t)\bigr) \cdot \d W_t\biggr)_T,
\end{equation*}
where $b_i(t, x) = b\bigl(t, x, \mu^i_{t \land T}, a(t, x)\bigr)$ for $i = 1$, $2$. Note that since $b$ is bounded by Assumption \ref{ass:control_problem} and $\sigma$ is invertible by Item \ref{it:nondeg} of the same assumption, the stochastic exponential is a martingale by Novikov's condition, so that the measure $\qr$ is indeed well-defined and equivalent to $\pr$. Hence, by Girsanov's theorem, the process $B$ defined by $B_t = W_t - \int_0^t \sigma^{-1}(b_2(s, X_s) - b_1(s, X_s)) \, \d s$ is a Brownian motion under $\qr$. Thus, $X$ solves the SDE
\begin{equation*}
    \d X_t = b\bigl(t, X_t, \mu^2_{t \land T}, a(t, X_t)\bigr) \, \d t + \sigma \, \d B_t
\end{equation*}
under $\qr$. This SDE exhibits uniqueness in law, whence $\L^{\qr}(X) = \nu^2$. Consequently, for any $t \in [0, T]$ and all Borel subset $A \subset C([0, \infty); \R^d)$, we have
\begin{align*}
    \bigl\lvert \nu^1_t(A) - \nu^2_t(A)\bigr\rvert = \bigl\lvert \ev\bigl[\mathbf{1}_{\{X_{\cdot \land t} \in A\}} (1 - Z_t)\bigr]\bigr\rvert \leq \sqrt{\nu^1_t(A)} \lVert Z_t - 1\rVert_{L^2} \leq \lVert Z_t - 1\rVert_{L^2},
\end{align*}
where $Z_t = \cal{E}\bigl(\int_0^{\cdot} \sigma^{-1} (b_2(s, X_s) - b_1(s, X_s)) \cdot \, \d W_s\bigr)_t$. Taking the supremum over all Borel sets $A$ implies that $\lVert \nu^1_t - \nu^2_t\rVert_{\textup{TV}} \leq \lVert Z_t - 1\rVert_{L^2}$. Recalling that $Z$ is a martingale started from $1$, we find
\begin{align} \label{eq:tv_bound}
    \lVert \nu^1_t - \nu^2_t\rVert_{\textup{TV}}^2 &\leq \lVert Z_t - 1\rVert_{L^2}^2 = \ev\lvert Z_t\rvert^2 - 2\ev Z_t + 1 = \ev\lvert Z_t\rvert^2 - 1 \notag \\
    &= \ev\int_0^t \bigl\lvert \sigma^{-1} \bigl(b_2(s, X_s) - b_1(s, X_s)\bigr) Z_s\bigr\rvert^2 \, \d s\notag \\
    &\leq c_{\sigma}^2 \ev \int_0^t \bigl\lvert b_2(s, X_s) - b_1(s, X_s)\bigr\rvert^2 \lvert Z_s\rvert^2 \, \d s \notag \\
    &\leq c_{\sigma}^2 C_b^2 \ev \lvert Z_T\rvert^2 \int_0^t \lVert \mu^1_s - \mu^2_s\rVert_{\textup{TV}}^2 \, \d s,
\end{align}
where we applied It\^o's isometry in the second line and used in the last inequality that $\ev \lvert Z_s\rvert^2 \leq \ev \lvert Z_T\rvert^2$ for $s \in [0, T]$. In addition, we used the Lipschitz continuity of $b$ in the measure argument and the invertibility of $\sigma$ guaranteed by Items \ref{it:coefficient} and \ref{it:nondeg} of Assumption \ref{ass:control_problem}. By Lemma \ref{lem:bounds}, we have that $\pr(\tau > T)$ and $\qr(\tau > T)$ are bounded from below by some $\delta > 0$ independent of $\mu^1$ and $\mu^2$. Consequently,
\begin{equation} \label{eq:conditional_tv_bound}
    \bigl\lVert \Phi(\mu^1)_t - \Phi(\mu^2)_t\bigr\rVert_{\textup{TV}} \leq \frac{2}{\delta} \bigl\lVert \pr(X_t \in \cdot, \, \tau > t) - \qr(X_t \in \cdot, \, \tau > t)\bigr\rVert_{\textup{TV}} \leq \frac{2}{\delta} \lVert \nu^1_t - \nu^2_t\rVert_{\textup{TV}}.
\end{equation}
Inserting this into \eqref{eq:tv_bound} yields
\begin{equation*} 
    \sup_{0 \leq s \leq t}\bigl\lVert \Phi(\mu^1)_s - \Phi(\mu^2)_s\bigr\rVert_{\textup{TV}}^2 \leq \frac{4}{\delta^2} \sup_{0 \leq s \leq t} \lVert \nu^1_s - \nu^2_s\rVert_{\textup{TV}}^2 \leq C \int_0^t \sup_{0 \leq u \leq s} \lVert \mu^1_u - \mu^2_u\rVert_{\textup{TV}}^2 \, \d s,
\end{equation*}
where $C = \frac{4 c_{\sigma}^2 C_b^2 \ev\lvert Z_T\rvert^2}{\delta^2}$. Multiplying both sides by $e^{-2Ct}$, integrating over $t \in [0, t']$, and applying integration by parts implies that
\begin{equation*}
    \int_0^{t'} e^{-2Ct} \sup_{0 \leq s \leq t \land T}\bigl\lVert \Phi(\mu^1)_s - \Phi(\mu^2)_s\bigr\rVert_{\textup{TV}}^2 \, \d t \leq \frac{1}{2} \int_0^{t'} e^{-2Ct} \sup_{0 \leq s \leq t \land T} \lVert \mu^1_s - \mu^2_s\rVert_{\textup{TV}}^2 \, \d t.
\end{equation*}
Hence, equipping $C([0, T]; \P(\R^d))$ with the complete metric
\begin{equation*}
    (m^1, m^2) \mapsto \biggl(\int_0^{\infty} e^{-2Ct} \sup_{0 \leq s \leq t \land T} \lVert m^1_s - m^2_s\rVert_{\textup{TV}}^2 \, \d s\biggr)^{1/2},
\end{equation*}
we find that $\Phi$ is a contraction. Consequently, there exists a unique fixed point $\mu$. It follows that the unique strong solution of SDE \eqref{eq:sde_mu_fixed} with this choice of $\mu$, is a strong solution to SDE \eqref{eq:conditional_mv_closed}. Since every solution to the latter equation yields a fixed point of $\Phi$, the contraction property of $\Phi$ implies weak uniqueness for the conditional McKean--Vlasov \eqref{eq:conditional_mv_closed}. Strong uniqueness is then a consequence of the result by Veretennikov \cite[Theorem 1]{veretennikov_strong_ex_1981}. This concludes the proof.
\end{proof}

Next, we show how starting from any admissible control we can construct a closed-loop control with an equal or greater reward. The construction is based on a measurable selection argument, for which we require Assumption \ref{ass:convex} to be in place. Further, we have to assume that the domain $D$ satisfies the \textit{Poincar\'e--Zaremba cone condition}. That is, for any point $x \in \partial D$, there exists an open cone $C$ with apex $x$ and an $\epsilon > 0$ such that $C \cap B_{\epsilon}(x) \subset \bar{D}^c$. This condition translates into a regularity property for the hitting time $\tau_{\alpha, \mu}$ of $X^{\alpha, \mu}$ on $\partial D$ and is relevant for the mimicking argument outlined in Section \ref{sec:introduction}.

\begin{theorem} \label{thm:closed_from_weak}
Let Assumptions \ref{ass:control_problem} and \ref{ass:convex} be satisfied and suppose that $D$ possesses the Poincar\'e--Zaremba cone condition. Then for any admissible control $(\alpha, \mu)$, there exists a closed-loop control $(\tilde{\alpha}, \mu)$ such that $J(\tilde{\alpha}, \mu) \geq J(\alpha, \mu)$. In particular, $V_{\textup{closed}} = V$.
\end{theorem}

\begin{proof}
Let $(\alpha, \mu)$ be a weak control, set $X = X^{\alpha, \mu}$, and define $\Lambda_t = \int_0^t \mathbf{1}_{\{X_s \notin \bar{D}\}} \, \d s$. Now, we proceed similarly as in the proof of Theorem 3.6 in \cite{haussmann_oc_1990} to extract a feedback function $\tilde{a}$ from $\alpha$, which depends on both $X_t$ and $\Lambda_t$. Define
\begin{equation*}
    c_h(t, x, \lambda, m) = \ev\bigl[h(t, X_t, m, \alpha_t) \bigr\vert X_t = x, \Lambda_t = \lambda\bigr]
\end{equation*}
for $(t, x, \lambda, m) \in [0, T] \times \R^d \times [0, \infty) \times \P(D)$ and $h \in \{b, f\}$. Then by the convexity property from Assumption \ref{ass:convex}, we have
\begin{equation*}
    (c_b, c_f)(t, x, \lambda, m) \in \Bigl\{\bigl(b(t, x, m, a), z\bigr) \define z \in \R,\, a \in A,\, z \leq f(t, x, m, a)\Bigr\}
\end{equation*}
for all $(t, x, \lambda, m) \in [0, T] \times \R^d \times [0, \infty) \times \P(D)$. Hence, additionally drawing on the continuity and upper semicontinuity of $b$ and $f$ in the control argument, guaranteed by Assumption \ref{ass:convex}, we may apply the measurable selection result \cite[Theorem A.9]{haussmann_oc_1990}, whereby there exists a measurable function $\tilde{a} \define [0, T] \times \R^d \times [0, \infty) \to A$ such that
\begin{align} \label{eq:meas_sel}
\begin{split}
    \ev\bigl[b(t, X_t, \mu_t, \alpha_t) \big\vert X_t, \Lambda_t\bigr] &= c_b(t, X_t, \Lambda_t, \mu_t) = b\bigl(t, X_t, \mu_t, \tilde{a}(t, X_t, \Lambda_t)\bigr) \\
    \ev\bigl[f(t, X_t, \mu_t, \alpha_t) \big\vert X_t, \Lambda_t\bigr] &= c_f(t, X_t, \Lambda_t, \mu_t) \leq f\bigl(t, X_t, \mu_t, \tilde{a}(t, X_t, \Lambda_t)\bigr)
\end{split}
\end{align}
for $\Leb \otimes \pr$-a.e.\@ $(t, \omega) \in [0, \infty) \times \Omega$. Note that the dependence on $\mu_t$ is absorbed into the time argument of $\tilde{a}$. Next, we apply the mimicking theorem by Brunick \& Shreve \cite[Corollary 3.7]{brunick_mimicking_2013} to the It\^o process $(X, \Lambda)$, using the fact that for $\Leb \otimes \pr$-a.e.\@ $(t, \omega) \in [0, \infty) \times \Omega$, we have
\begin{equation*}
    \ev\biggl[
    \begin{pmatrix}
        b(t, X_t, \mu_t, \alpha_t)\\
        \mathbf{1}_{\{X_t \notin \bar{D}\}}
    \end{pmatrix}
    \biggr\vert X_t, \Lambda_t\biggr] = 
    \begin{pmatrix}
        b\bigl(t, X_t, \mu_t, \tilde{a}(t, X_t, \Lambda_t)\bigr)\\
        \mathbf{1}_{\{X_t \notin \bar{D}\}}
    \end{pmatrix}.
\end{equation*}
This yields a probability space $(\tilde{\Omega}, \tilde{\F}, \tilde{\pr})$ carrying a $d$-dimensional Brownian motion $\tilde{W}$ and a continuous $\R^{d + 1}$-valued process $(\tilde{X}, \tilde{\Lambda})$ satisfying the SDE
\begin{equation} \label{eq:sde_mimicking}
    \d \tilde{X}_t = b\bigl(t, \tilde{X}_t, \mu_t, \tilde{a}(t, \tilde{X}_t, \tilde{\Lambda}_t)\bigr) \, \d t + \sigma \, \d \tilde{W}_t, \quad \d \tilde{\Lambda}_t = \mathbf{1}_{\{\tilde{X}_t \notin \bar{D}\}} \, \d t,
\end{equation}
such that $\L(\tilde{X}_t, \tilde{\Lambda}_t) = \L(X_t, \Lambda_t)$ for all $t \in [0, T]$. Next, we claim that $\tau = \inf\{t > 0 \define \Lambda_t > 0\}$ holds $\pr$-a.s.\@ and, similarly, $\tilde{\tau} = \inf\{t > 0 \define \tilde{\Lambda}_t > 0\}$ holds $\tilde{\pr}$-a.s.\@, where $\tilde{\tau} = \inf\{t > 0 \define \tilde{X}_t \notin D\}$. We will only show this for the process $(X, \Lambda)$, the proof for $(\tilde{X}, \tilde{\Lambda})$ is analogous. Since $X_t \in \bar{D}$ on $[0, \tau]$, it holds that $\Lambda_t = 0$ on $[0, \tau]$, whence $\tau \leq \inf\{t > 0 \define \tilde{\Lambda}_t > 0\}$. Let us now establish the converse. By Proposition \ref{prop:hitting_boundary}, $\pr$-a.s.\@ for every $\epsilon > 0$ there exists $t \in (0, \epsilon]$ such that $X_{\tau + t} \notin \bar{D}$. By continuity of $X$, this means that $\Lambda_{\tau + \epsilon} > 0$. Hence, it follows that $\inf\{t > 0 \define \Lambda_t > 0\} \leq \tau + \epsilon$. Letting $\epsilon \to 0$ yields the desired inequality, so that $\pr$-a.s.\@ $\tau = \inf\{t > 0 \define \Lambda_t > 0\}$. This implies that $\mu_t = \L(X_t \vert \tau > 0) = \L(X_t \vert \Lambda_t = 0) = \L(\tilde{X}_t \vert \tilde{\Lambda}_t = 0) = \L(\tilde{X}_t \vert \tilde{\tau} > t)$, so that $\tilde{X}$ is a weak solution to the conditional McKean--Vlasov SDE \eqref{eq:conditional_mv}, where $\alpha_t$ is replaced by $\tilde{a}(t, \tilde{X}_t, \tilde{\Lambda}_t)$. Moreover, in view of Equation \eqref{eq:meas_sel}, we have
\begin{align} \label{eq:cost_lower}
    J(\alpha, \mu) &= \int_0^T \ev\bigl[f(t, X_t, \mu_t, \alpha_t) \big\vert \tau > t \bigr] \, \d t + g(\mu_T) \notag \\
    &\leq \int_0^T \ev\bigl[f\bigl(t, X_t, \mu_t, \tilde{a}(t, X_t, \Lambda_t)\bigr) \big\vert \Lambda_t = 0 \bigr] \, \d t + g(\mu_T) \notag \\
    &= \int_0^T \ev^{\tilde{\pr}}\bigl[f\bigl(t, \tilde{X}_t, \mu_t, \tilde{a}(t, \tilde{X}_t, \tilde{\Lambda}_t)\bigr)\big\vert \tilde{\Lambda}_t = 0\bigr] \, \d t + g(\mu_T) \notag \\
    &= \int_0^T \ev^{\tilde{\pr}}\bigl[f\bigl(t, \tilde{X}_t, \mu_t, \tilde{a}(t, \tilde{X}_t, 0)\bigr)\big\vert \tilde{\tau} > t\bigr] \, \d t + g(\mu_T).
\end{align}
Let us now define $a \define [0, \infty) \times \R^d \to A$ by $a(t, x) = \tilde{a}(t, x, 0)$ and let $X'$ be the unique strong solution to SDE
\begin{equation*}
    \d X'_t = b\bigl(t, X'_t, \mu_t, a(t, X'_t)\bigr) \, \d t + \sigma \, \d W_t
\end{equation*}
with initial condition $X'_0 = \xi$. Note that $X'$ and $\tilde{X}$ satisfy the same SDE up to their respective first hitting time of $D^c$, so uniqueness in law for this SDE implies that $\L(X'_{\cdot \land \tau'}) = \L(\tilde{X}_{\cdot \land \tilde{\tau}})$, where $\tau' = \inf\{t > 0 \define X'_t \notin D\}$. Consequently, $\mu_t = \L(\tilde{X}_t \vert \tilde{\tau} > t) = \L(X'_t \vert \tau' > t)$, so that $X'$ is a strong solution to the conditional McKean--Vlasov SDE \eqref{eq:conditional_mv} with $\alpha$ replaced by $\tilde{\alpha}$ given by $\tilde{\alpha}_t = a(t, X'_t)$. Thus, $(\tilde{\alpha}, \mu)$ is a closed-loop control and by \eqref{eq:cost_lower} it holds that
\begin{align*}
    J(\alpha, \mu) &\leq \int_0^T \ev^{\tilde{\pr}}\bigl[f\bigl(t, \tilde{X}_t, \mu_t, \tilde{a}(t, \tilde{X}_t, 0)\bigr)\big\vert \tilde{\tau} > t\bigr] \, \d t + g(\mu_T) \\
    &= \int_0^T \ev\bigl[f\bigl(t, X'_t, \mu_t, a(t, X'_t)\bigr)\big\vert \tau' > t\bigr] \, \d t + g(\mu_T) \\
    &= J(\tilde{\alpha}, \mu).
\end{align*}
\end{proof}

\begin{remark}
Note that the results in this section also hold if we no longer require that $b$ is bounded but instead assume that $b(t, x, m, a) = a$ and $\ev \int_0^T \lvert \alpha_t\rvert^2 \, \d t < \infty$. This is one of the settings considered by Lions \cite{lions_cond_proc_2016}. We may achieve this by first considering the case of bounded controls and then passing to the limit. In principle, we could also follow this procedure in the case of a more general unbounded $b$ that satisfies suitable growth assumptions, but it becomes more involved if $b$ can depend on the conditional law. We impose the boundedness assumptions on $b$ to avoid these technical issues.
\end{remark}

\begin{remark} \label{rem:comp_car}
The counterpart to Theorem \ref{thm:closed_from_weak} by Carmona \& Lacker \cite{carmona_mimicking_conditional_2024} is established under slightly stronger assumptions, taking the boundary $\partial D$ to be smooth instead of having $D$ satisfy the weaker Poincar\'e--Zaremba cone condition. Moreover, \cite{carmona_mimicking_conditional_2024} only covers the linear setting $b(t, x, m, a) = a$, so, in particular, does not allow for McKean--Vlasov dynamics. However, the proof techniques from \cite{carmona_mimicking_conditional_2024} can likely be extended to our setting without substantial difficulties. 
\end{remark}

\section{Fleming--Viot Dynamics of McKean--Vlasov Type}

In this section, we establish a correspondence between the conditional McKean--Vlasov SDE \eqref{eq:conditional_mv} and Fleming--Viot dynamics of McKean--Vlasov type. We consider an $\R^d$-valued process $(X_t)_{t \geq 0}$ that follows the dynamics
\begin{equation} \label{eq:fleming_viot}
    \d X_t = b\bigl(t, X_t, \mu_t, a(t, X_t)\bigr) \, \d t + \sigma \, \d W_t + \d J_t
\end{equation}
with initial condition $X_0 = \xi$ and $\mu_t = \L(X_t)$. The jump process $J = (J_t)_{t \geq 0}$ is given by $J_t = \sum_{k = 1}^{\infty} \int_0^t (\xi_k - X_{s-}) \, \d \mathbf{1}_{\{\tau_k \leq s\}}$, where the stopping times $(\tau_k)_{k \geq 0}$ are defined by $\tau_0 = 0$ and $\tau_k = \inf\{t > \tau_{k - 1} \define X_{t-} \notin D\}$ for $k \geq 1$. The $D$-valued random variables $(\xi_k)_{k \geq 1}$ satisfy $\L(\xi_k \vert \tau_k = t) = \mu_t$ for $t \geq 0$. Note that there are two subtleties here. Namely, it is not immediately clear whether $\pr(\tau_k = t) = 0$ for all $t \geq 0$, $k \geq 1$ and whether $\tau_{\infty} = \lim_{k \to \infty} \tau_k = \infty$ almost surely. Should either of these properties be violated, McKean--Vlasov SDE \eqref{eq:fleming_viot} is not well-defined. Indeed, if $\pr(\tau_k = t) > 0$ for some $t \geq 0$ and $k \geq 1$, then to avoid a circularity in the definition of the law of $\xi_k$, we should require that $\L(\xi_k \vert \tau_k = t) = \mu_{t-}$ instead of $\L(\xi_k \vert \tau_k = t) = \mu_t$. But with this construction, the representative particle will be reinserted at a location on the boundary $\partial D$ with positive probability, in which case it would immediately jump again. Next, if $\tau_{\infty} < \infty$, then $X_s$ is not defined on all of $\Omega$ for all times $s \geq 0$, so we cannot make sense of the law $\mu_s$ of $X_s$. To circumvent these issues, when we speak of a solution to McKean--Vlasov SDE \eqref{eq:fleming_viot}, we shall always require that $\pr(\tau_k = t) = 0$ for all $t \geq 0$, $k \geq 1$ and $\tau_{\infty} = \infty$ almost surely.

The process $X$ is a mean-field analogue of the Fleming--Viot type particle system introduced by Burdzy et al.\@ \cite{burdzy_flem_viot_1996}. Whenever the representative particle hits the boundary, it is reinserted at a location that is distributed according to the prevailing law $\mu_t$ of the system. As shown below, the law of the solution of \eqref{eq:fleming_viot} coincides with the conditional law of the killed process discussed in Section \ref{sec:equivalence}. First, however, let us clarify how weak and strong solutions for McKean--Vlasov SDE \eqref{eq:fleming_viot} are defined and discuss different notions of uniqueness. This requires some care, as we have to pin down the randomness introduced by the reinsertions. For this we rely on a specific construction of the sequence $(\xi_k)_{k \geq 1}$, which we outline next. We say that a measurable function $H \define \P(\R^d) \times [0, 1] \to \R^d$ is a \textit{representation} of $\P(\R^d)$ if $H(m, \cdot)^{\#}\cal{U}([0, 1]) = m$ for all $m \in \P(\R^d)$. Here the symbol $\#$ in the superscript indicates the pushforward of a measurable function and $\cal{U}([0, 1])$ is the uniform distribution on $[0, 1]$. Representations of $\P(\R^d)$ are guaranteed to exist by \cite[Theorem 1]{blackwell_skorokhod_1983}. Now, if $(U_k)_{k \geq 1}$ is an i.i.d.\@ sequence of random variables, uniformly distributed on $[0, 1]$ and such that $\xi$, $W$, and $U_k$, $k \geq 1$, are independent, then setting $\xi_k = H(\mu_{\tau_k}, U_k)$, we find
\begin{equation*}
    \L(\xi_k \vert \tau_k = t) = \L\bigl(H(\mu_{\tau_k}, U_k) \big\vert \tau_k = t\bigr) = \L(H(\mu_t, U_k)) = \mu_t.
\end{equation*}

With this at hand, we can state our definitions of existence and uniqueness. For convenience, we call $(\Omega, \F, \bb{F}, \pr, W , (U_k)_{k \geq 1})$ a \textit{stochastic setup} if $(\Omega, \F, \bb{F}, \pr)$ is a filtered probability space carrying a $d$-dimensional $\bb{F}$-Brownian motion $W$ and a sequence $(U_k)_{k \geq 1}$ of i.i.d.\@ random variables uniformly distributed on $[0, 1]$ such that $\F_0$, $W$, and $(U_k)_{k \geq 1}$ are independent.

\begin{definition} \label{def:weak_solution}
A \textit{weak solution} of McKean--Vlasov SDE \eqref{eq:fleming_viot} with underlying representation $H$ of $\P(\R^d)$ is a stochastic setup $(\Omega, \F, \bb{F}, \pr, W, (U_k)_{k \geq 1})$ carrying an $\R^d$-valued c\`adl\`ag $\bb{F}$-adapted stochastic process $X$ such that 
\begin{enumerate}
    \item $(W, (X_t)_{t \in [0, \tau_k)})$ is independent of $(U_{\ell})_{\ell \geq k}$ for $k \geq 1$;
    \item $X_{\tau_k} = H(\mu_{\tau_k}, U_k)$ for $k \geq 1$;
    \item SDE \eqref{eq:fleming_viot} holds and $\mu_t = \L(X_t)$ for all $t \geq 0$.
\end{enumerate}
\end{definition}

Note that the random variables $\xi$, $J$, $\tau_k$, and $\xi_k$, $k \geq 1$, implicitly or explicitly appearing in Definition \ref{def:weak_solution} are determined by $X$.

\begin{definition} \label{def:strong_solution}
A \textit{strong solution} of McKean--Vlasov SDE \eqref{eq:fleming_viot} with underlying representation $H$ of $\P(\R^d)$ on a given stochastic setup $(\Omega, \F, \bb{F}, \pr, W, (U_k)_{k \geq 1})$, carrying an $\R^d$-valued $\F_0$-measurable random variable $\xi$, is an $\R^d$-valued c\`adl\`ag stochastic process $X$ such that
\begin{enumerate}
    \item $(X_t)_{t \in [0, \tau_k)}$ is $(\sigma(\xi, W_s, U_{\ell} \define s \in [0, t], 1 \leq \ell \leq k - 1))_{t \geq 0}$-adapted for all $k \geq 1$;
    \item $X_{\tau_k} = H(\mu_{\tau_k}, U_k)$ for $k \geq 1$;
    \item SDE \eqref{eq:fleming_viot} holds and $\mu_t = \L(X_t)$ for all $t \geq 0$.
\end{enumerate}
\end{definition}

Lastly, we state the our notions of pathwise uniqueness and uniqueness in law.

\begin{definition} \label{def:uniqueness}
We say that \textit{pathwise uniqueness} holds if any two weak solutions of McKean--Vlasov SDE \eqref{eq:fleming_viot} with the same underlying representation of $\P(\R^d)$ defined on the same stochastic setup coincide almost surely.

We say that \textit{uniqueness in law} holds if any two weak solutions of McKean--Vlasov SDE \eqref{eq:fleming_viot} with possibly distinct underlying representations of $\P(\R^d)$ have the same law.
\end{definition}

The terminology ``pathwise uniqueness'' and ``uniqueness in law'' is justified by a theorem of Yamada--Watanabe type, Theorem \ref{thm:yw} from Appendix \ref{app:yw}, whereby pathwise uniqueness implies uniqueness in law. In Theorem \ref{thm:correspondence} below, we establish pathwise uniqueness for McKean--Vlasov SDE \eqref{eq:fleming_viot}, which in view of Theorem \ref{thm:yw}, implies uniqueness in law. Note that if we admit a broader class of weak solutions, for which the sequence $(\xi_k)_{k \geq 1}$ is not necessarily constructed from a representation of $\P(\R^d)$ and independent uniform random variables, then uniqueness in law within this broader class no longer holds. A detailed discussion of this can be found in Appendix \ref{app:yw} and Example \ref{ex:counter} therein, highlighting the crucial role played by our explicit construction of $(\xi_k)_{k \geq 1}$.

\begin{theorem} \label{thm:correspondence}
Let Assumption \ref{ass:control_problem} be satisfied. Then, for any representation of $\P(\R^d)$, McKean--Vlasov SDE \eqref{eq:fleming_viot} has a strong solution $X$ and pathwise uniqueness holds. Moreover, we have $\L(X_t) = \L(X'_t \vert \tau' > t)$ for all $t \geq 0$, where $X'$ is the unique strong solution of McKean--Vlasov SDE \eqref{eq:conditional_mv_closed} and $\tau' = \inf\{t > 0 \define X'_t \notin D$\}.
\end{theorem}

\begin{proof}
Let $X'$ be the unique strong solution of the conditional McKean--Vlasov SDE \eqref{eq:conditional_mv_closed}, guaranteed to exist by Proposition \ref{prop:exist_unique}, and set $\mu_t = \L(X'_t \vert \tau' > t)$, where $\tau' = \inf\{t > 0 \define X'_t \notin D$\}. We will use the flow $\mu = (\mu_t)_{t \geq 0}$ to construct a strong solution $X$ to McKean--Vlasov SDE \eqref{eq:fleming_viot} for a given representation $H$ of $\P(\R^d)$. In fact, we simply obtain $X$ as a strong solution of the SDE \eqref{eq:fleming_viot}, where we view the appearance of $\mu_t$ both in the coefficient $b$ and to describe the conditional law $\L(\xi_k \vert \tau_k = t)$ of the random variables $\xi_k$ as a fixed input rather than being determined by the SDE. Note that a strong solution to this SDE can be constructed in an iterative fashion, but it is a priori only defined up to the time $\tau_{\infty} = \lim_{k \to \infty} \tau_k \in [0, \infty]$. Indeed, we first solve SDE \eqref{eq:fleming_viot} up to the stopping time $\tau_1 = \inf\{t > 0 \define X_{t-} \notin D\}$. Then, we set $\xi_1 = H(\mu_{\tau_1}, U_1)$, insert the jump from $X_{\tau_1-}$ to $\xi_k$ at time $\tau_1$, and solve SDE \eqref{eq:fleming_viot} from $\tau_1$ onwards up to the stopping time $\tau_2 = \inf\{t > \tau_1 \define X_{t-} \notin D\}$. We iterate this procedure indefinitely, yielding the desired solution. By Proposition \ref{prop:hitting_no_atoms}, we have that $\pr(\tau_k = t) = 0$ for all $t \geq 0$ and $k \geq 1$. Note further that the constructed solution is strong in the sense of Definition \ref{def:strong_solution} (except, of course, for the fact that we have not verified $\L(X_t) = \mu_t$ yet). Moreover, pathwise uniqueness in the usual sense holds on the segments $[\tau_k, \tau_{k + 1})$, $k \geq 0$, where $\tau_0 = 0$, and the random variables $\xi_k$ are uniquely determined by $U_k$ and the representation $H$. Consequently, pathwise uniqueness as stated in Definition \ref{def:uniqueness} holds for SDE \eqref{eq:fleming_viot} with $(\mu_t)_{t \geq 0}$ as a fixed input among solutions with the underlying representation $H$ of $\P(\R^d)$. 

Our next goal will be to show that $\tau_{\infty} = \infty$ almost surely. To that end, let us define the function $F^n_t = \sum_{k = 1}^n \pr(\tau_k \leq t)$ for $n \geq 1$ as well as $F_t = \sum_{k = 1}^{\infty} \pr(\tau_k \leq t)$. It holds for $k \geq 2$ that
\begin{align} \label{eq:conv_for_prob}
    \pr(\tau_k \leq t) &= \ev[\pr(\tau_k \leq t \vert \tau_{k - 1})] \notag \\
    &= \int_0^t \pr(\tau_k - \tau_{k - 1} \leq t - s \vert \tau_{k - 1} = s) \, \d \pr(\tau_{k - 1} \leq s) \notag \\
    &= \int_0^t \pr(\tau_{s, \mu_s} \leq t - s) \, \d \pr(\tau_{k - 1} \leq s).
\end{align}
Here for $s \geq 0$ and $m \in \P(D)$, $X^{s, m}$ denotes the unique strong solution to
\begin{equation} \label{eq:restarted}
    \d X^{s, m}_u = b\bigl(s + u, X^{s, m}_u, \mu_{s + u}, a(s + u, X^{s, m}_u)\bigr) \, \d t + \sigma \, \d (W_{s + u} - W_s)
\end{equation}
with initial condition $X^{s, m}_0 \sim m$ and $\tau_{s, m} = \inf\{u > 0 \define X^{s, m}_u \notin D\}$. (Note that this should be distinguished from the process $X^{\alpha, \mu}$ that we introduced in Section \ref{sec:equivalence}.) From \eqref{eq:conv_for_prob}, we deduce
\begin{align} \label{eq:conv_for_finite_firing}
    F^n_t &= \pr(\tau_1 \leq t) + \sum_{k = 2}^n \int_0^t \pr(\tau_{s, \mu_s} \leq t - s) \, \d \pr(\tau_{k - 1} \leq s) \notag \\
    &= \pr(\tau_1 \leq t) + \int_0^t \pr(\tau_{s, \mu_s} \leq t - s) \, \d F^{n - 1}_s
\end{align}
for $n \geq 2$. 

Next, we claim that for any $t \geq 0$ there exists $p_t \in [0, 1)$ such that $\pr(\tau_{s, \mu_s} \leq t - s) \leq p_t$ for $s \in [0, t]$. If this is indeed the case, then it follows from the above equality that $F^n_t \leq \pr(\tau_1 \leq t) + p_t F^{n - 1}_t \leq \pr(\tau_1 \leq t) + p_t F^n_t$. Rearranging gives $F^n_t \leq \frac{1}{1 - p_t} \pr(\tau_1 \leq t)$, so that $F_t = \lim_{n \to \infty} F^n_t \leq \frac{1}{1 - p_t} \pr(\tau_1 \leq t) < \infty$ for all $t \geq 0$. Moreover, taking the limit $n \to \infty$ in \eqref{eq:conv_for_finite_firing} yields
\begin{equation} \label{eq:conv_for_firing}
    F_t = \pr(\tau_1 \leq t) + \int_0^t \pr(\tau_{s, \mu_s} \leq t - s) \, \d F_s.
\end{equation}
Let us prove the claim. Fix $t \geq 0$. First note that since $\mu$ is continuous in the topology of weak convergence, we can find $\epsilon > 0$ such that $\mu_s(D_{\epsilon}) \geq \epsilon$ for all $s \in [0, T]$, where $D_{\delta} = \{x \in D \define d(x, \partial D) > \delta\}$ for $\delta > 0$. Now, we set $A_{t, s} = \{\tau_{s, \mu_s} \leq t - s\}$ and $X^s = X^{s, \mu_s}$, and decompose
\begin{align*}
    \pr(\tau_{s, \mu_s} \leq t - s) &= \pr\bigl(\tau_{s, \mu_s} \leq t - s \big\vert X^s_0 \in D_{\epsilon}\bigr) \mu_s(D_{\epsilon}) + \pr\bigl(\tau_{s, \mu_s} \leq t - s \big\vert X^s_0 \notin D_{\epsilon}\bigr) \mu_s(D_{\epsilon}^c) \\
    &\leq \pr\bigl(\tau_{s, \mu_s} \leq t - s \big\vert X^s_0 \in D_{\epsilon}\bigr) \mu_s(D_{\epsilon}) + \mu_s(D_{\epsilon}^c) \\
    &= 1 - \bigl(1 - \pr\bigl(\tau_{s, \mu_s} \leq t - s \big\vert X^s_0 \in D_{\epsilon}\bigr)\bigr)\mu_s(D_{\epsilon}) \\
    &\leq 1 - \epsilon\bigl(1 - \pr\bigl(\tau_{s, \mu_s} \leq t - s \big\vert X^s_0 \in D_{\epsilon}\bigr)\bigr).
\end{align*}
Hence, if we can show that there is $q_t \in [0, 1)$ with  $\pr(\tau_{s, \mu_s} \leq t - s \vert X^s_0 \in D_{\epsilon}) \leq q_t$ for all $s \in [0, t]$, then setting $p_t = 1 - \epsilon(1 - q_t)$, the claim follows. However, it holds that 
\begin{align*}
    \pr\bigl(\tau_{s, \mu_s} \leq t - s \big\vert X^s_0 \in D_{\epsilon}\bigr) &= \frac{1}{\mu_s(D_{\epsilon})}\int_{D_{\epsilon}} \pr(\tau_{s, \mu_s} \leq t - s \vert X^s_0 = x) \, \d \mu_s(x) \\
    &= \frac{1}{\mu_s(D_{\epsilon})}\int_{D_{\epsilon}} \pr(\tau_{s, \delta_x} \leq t - s) \, \d \mu_s(x)
\end{align*}
For any $x \in D_{\epsilon}$, we can bound $\pr(\tau_{s, \delta_x} \leq t - s)$ by the probability that $X^{s, \delta_x}$ exits $B_{\epsilon}(x) \subset D$ before time $t - s$. The latter can be shown to be bounded by $q_t \in [0, 1)$ independently of $s \in [0, t]$. This concludes the proof of the claim. Note that the function $t \mapsto \pr(\tau_{s, \mu_s} \leq t - s)$ is nondecreasing, so we may choose $p_t$ such that $t \mapsto p_t$ is nondecreasing and right-continuous.

Next, we will show that \eqref{eq:conv_for_firing} also holds for $-\log \pr(\tau' > t)$ in place of $F_t$. The processes $X$ and $X'$ solve the same SDE up to their respective first hitting times $\tau_1$ and $\tau'$ of $D$. Hence, for $0 \leq s \leq t$, we have
\begin{equation} \label{eq:exp_identity}
    \pr(\tau' > t) = \pr(\tau_1 > t) = \pr(\tau_1 > s) \pr(\tau_1 > t \vert \tau_1 > s) = \pr(\tau' > s) \pr(\tau_{s, \mu_s} > t - s).
\end{equation}
Rearranging gives $-1 = -\frac{\pr(\tau' > t)}{\pr(\tau' > s)} - \pr(\tau_{s, \mu_s} \leq t - s)$. We divide by $\pr(\tau' > s)$ on both sides and integrate over $s \in [0, t]$ with respect to $s \mapsto \pr(\tau' > s)$, whence
\begin{align} \label{eq:conv_for_log_prob}
    -\log\pr(\tau' > t) &= - \int_0^t \frac{\pr(\tau' > t)}{\pr(\tau' > s)^2} \, \d \pr(\tau' > s) - \int_0^t \pr(\tau_{s, \mu_s} \leq t - s) \, \d \log \pr(\tau' > s) \notag \\
    &= \pr(\tau' \leq t) - \int_0^t \pr(\tau_{s, \mu_s} \leq t - s) \, \d \log \pr(\tau' > s)
\end{align}
as required. Now, if we can show that the convolution equation $f_t = \pr(\tau' \leq t) + \int_0^t \pr(\tau_{s, \mu_s} \leq t - s) \, \d f_s$ has a unique solution, then it follows from \eqref{eq:conv_for_firing} and \eqref{eq:conv_for_log_prob} that $F_t = -\log \pr(\tau' > t)$. We shall establish this next. Let $L$ be the space of nondecreasing c\`adl\`ag functions $f \define [0, \infty) \to \R$ such $f_t \leq \frac{1}{1 - p_t} \pr(\tau_1 \leq t)$ for $t \geq 0$. Since $t \mapsto p_t$ is nondecreasing and right-continuous, the space $L$ is a complete lattice with the partial order $\leq$ defined by $f \leq g$ if $f_t \leq g_t$ for all $t \in [0, \infty)$ and the meet and join of a subset $A \subset L$ given by the right-continuous modifications of the nondecreasing functions $t \mapsto \inf_{f \in A} f_t$ and $t \mapsto \sup_{f \in A} f_t$.
Moreover, the function from $L$ to $L$ that maps $f \in L$ to $t \mapsto \pr(\tau_1 \leq t) + \int_0^t \pr(\tau_{s, \mu_s} \leq t - s) \, \d f_s$ is monotonic. Thus, the Tarski fixed-point theorem guarantees the existence of a minimal fixed point $f^{\ast} \in L$. Clearly, this fixed point solves the equation $f^{\ast}_t = \pr(\tau_1 \leq t) + \int_0^t \pr(\tau_{s, \mu_s} \leq t - s) \, \d f^{\ast}_s$. Assume that $f$ is another solution, so that by minimality of $f^{\ast}$, we have $\Delta_t = f_t - f^{\ast}_t \geq 0$ for any $t \geq 0$. Then we find
\begin{equation*}
    \Delta_t = \int_0^t \pr(\tau_{s, \mu_s} \leq t - s) \, \d \Delta_s \leq p_t \Delta_t.
\end{equation*}
Since $p_t \in [0, 1)$, this can only be the case if $\Delta_t = 0$, so $f^{\ast}$ is indeed the unique solution to the convolution equation.

Let us now use an idea similar to \eqref{eq:conv_for_prob} to derive a convolution identity for $\L(X_t)$. We have for any $k \geq 1$ that
\begin{align*}
    \pr\bigl(X_t \in \cdot,\, t \in [\tau_k, \tau_{k + 1})\bigr) &= \int_0^t \pr\bigl(X_t \in \cdot,\, t \in [\tau_k, \tau_{k + 1}) \big\vert \tau_k = s\bigr) \, \d \pr(\tau_k \leq s) \\
    & = \int_0^t \pr\bigl(X_t \in \cdot,\, \tau_{k + 1} - \tau_k > t - s \big\vert \tau_k = s\bigr) \, \d \pr(\tau_k \leq s) \\
    &= \int_0^t \pr\bigl(X^{s, \mu_s}_{t - s} \in \cdot,\, \tau_{s, \mu_s} > t - s\bigr) \, \d \pr(\tau_k \leq s).
\end{align*}
Consequently, we obtain that
\begin{align} \label{eq:conv_for_law}
    \L(X_t) &= \pr(X_t \in \cdot,\, \tau_1 > t) + \sum_{k = 1}^{\infty} \pr(X_t \in \cdot,\, t \in [\tau_k, \tau_{k + 1})) \notag \\
    &= \pr(X_t \in \cdot,\, \tau_1 > t) + \int_0^t \pr\bigl(X^{s, \mu_s}_{t - s} \in \cdot,\, \tau_{s, \mu_s} > t - s\bigr) \, \d F_s.
\end{align}
An analogous identity can be deduced for $\mu_t$. Indeed, similarly to \eqref{eq:exp_identity}, we have
\begin{equation} \label{eq:extended_equality}
    \mu_t = \frac{\pr(\tau' > s)}{\pr(\tau' > t)} \pr\bigl(X'_t \in \cdot,\, \tau' > t \big\vert \tau' > s\bigr) = \frac{\pr(\tau' > s)}{\pr(\tau' > t)} \pr\bigl(X^{s, \mu_s}_{t - s} \in \cdot,\, \tau_{s, \mu_s} > t - s\bigr)
\end{equation}
for $s \in [0, t]$. Here we used in the second equality that conditional on $\{\tau' > s\}$, the process $(X'_{s + u})_{u \geq 0}$ is a solution to SDE \eqref{eq:restarted} with initial condition distributed according to $\mu_s = \L(X'_s \vert \tau' > s)$. Uniqueness in law of SDE \eqref{eq:restarted} therefore implies $\L((X'_u)_{u \geq s} \vert \tau' > s) = \L((X^{s, \mu_s}_u)_{u \geq 0})$, which gives the second equality in \eqref{eq:extended_equality}. Then we can proceed in the same fashion as below \eqref{eq:exp_identity}. We multiply both sides of \eqref{eq:extended_equality} by $\frac{\pr(\tau' > t)}{\pr(\tau' > s)^2}$ and integrate over $s \in [0, t]$ with respect to $s \mapsto \pr(\tau' > s)$. On the left-hand side, we have
\begin{equation*}
    \mu_t \int_0^t \frac{\pr(\tau' > t)}{\pr(\tau' > s)^2} \, \d \pr(\tau' > s) = - \mu_t \bigl(1 - \pr(\tau' > t)\bigr) = -\mu_t + \pr(X'_t \in \cdot,\, \tau' > t),
\end{equation*}
while on the right-hand side, we get
\begin{equation*}
    \int_0^t \frac{\pr\bigl(X^{s, \mu_s}_{t - s} \in \cdot,\, \tau_{s, \mu_s} > t - s\bigr)}{\pr(\tau' > s)} \, \d \pr(\tau' > s) = \int_0^t \pr\bigl(X^{s, \mu_s}_{t - s} \in \cdot,\, \tau_{s, \mu_s} > t - s\bigr) \, \d \log \pr(\tau' > s).
\end{equation*}
Using that $\pr(X'_t \in \cdot,\, \tau' > t) = \pr(X_t \in \cdot,\, \tau_1 > t)$ because $X'$ and $X$ solve the same unique in law SDE up to the first hitting time of $\partial D$, and rearranging yields
\begin{equation*}
    \mu_t = \pr(X_t \in \cdot,\, \tau_1 > t) - \int_0^t \pr\bigl(X^{s, \mu_s}_{t - s} \in \cdot,\, \tau_{s, \mu_s} > t - s\bigr) \, \d \log \pr(\tau' > s).
\end{equation*}
Since $-\log \pr(\tau' > s) = F_s$, the right-hand side above coincides with the right-hand side of Equation \eqref{eq:conv_for_law}, so we can conclude that $\L(X_t) = \mu_t = \L(X'_t \vert \tau' > t)$. Hence, we have constructed a strong solution $X$ to McKean--Vlasov SDE \eqref{eq:fleming_viot} for which $\L(X_t) = \L(X'_t \vert \tau' > t)$. 

It remains to establish pathwise uniqueness for McKean--Vlasov SDE \eqref{eq:fleming_viot} again understood as in Definition \ref{def:uniqueness}. We have already seen that when we view $(\mu_t)_{t \geq 0}$ as a fixed input, SDE \eqref{eq:fleming_viot} satisfies pathwise uniqueness in this sense. Thus, to show that pathwise uniqueness holds for McKean--Vlasov SDE \eqref{eq:fleming_viot}, it is enough to prove uniqueness in marginal law, i.e.\@ the flow $(\mu_t)_{t \geq 0}$ of marginal laws is the same for all weak solutions of McKean--Vlasov SDE \eqref{eq:fleming_viot}. To achieve this, we can simply invert the sequence of arguments from above, which only used the properties postulated for weak solutions in Definition \ref{def:weak_solution}. For any weak solution $\tilde{X}$ to McKean--Vlasov SDE \eqref{eq:fleming_viot} with corresponding hitting times $\tilde{\tau_k}$, $k \geq 1$, and law $\tilde{\mu}_t = \L(\tilde{X}_t)$ for $t \geq 0$, it follows that $\tilde{F}_t = \sum_{k = 1}^{\infty} \pr(\tilde{\tau}_k \leq t) = -\log \pr(\tilde{\tau}_1 > t)$. Then we proceed as below Equation \eqref{eq:conv_for_law} to show that
\begin{align*}
    \L(\tilde{X}_t \vert \tau_1 > t) &= \pr(\tilde{X}_t \in \cdot,\, \tilde{\tau}_1 > t) - \int_0^t \pr\bigl(X^{s, \tilde{\mu}_s}_{t - s} \in \cdot,\, \tau_{s, \tilde{\mu}_s} > t - s\bigr) \, \d \log \pr(\tilde{\tau}_1 > s) \\
    &= \pr(\tilde{X}_t \in \cdot,\, \tilde{\tau}_1 > t) + \int_0^t \pr\bigl(X^{s, \tilde{\mu}_s}_{t - s} \in \cdot,\, \tau_{s, \tilde{\mu}_s} > t - s\bigr) \, \d F_s \\
    &= \L(\tilde{X}_t).
\end{align*}
Now, let $\tilde{X}'$ be the unique strong solution to
\begin{equation*}
    \d \tilde{X}'_t = b\bigl(t, \tilde{X}'_t, \L(\tilde{X}_t), a(t, \tilde{X}'_t)\bigr) \, \d t + \sigma \, \d W_t
\end{equation*}
with initial condition $\tilde{X}'_0 = \xi$. By uniqueness in law for classical SDEs, the distributions of $(\tilde{X}'_{t \land \tilde{\tau}'})_{t \geq 0}$, where $\tilde{\tau}' = \inf\{t > 0 \define \tilde{X}'_t \notin D\}$, and $(\tilde{X}_{t \land \tilde{\tau}_1})_{t \geq 0}$ coincide, which means that $\L(\tilde{X}_t) = \L(\tilde{X}_t \vert \tilde{\tau}_1 > t) = \L(\tilde{X}'_t \vert \tilde{\tau}' > t)$. Consequently, $\tilde{X}'$ is a weak solution to the conditional McKean--Vlasov SDE \eqref{eq:conditional_mv_closed} which exhibits uniqueness in law by Proposition \ref{prop:exist_unique}. Hence, $\L(\tilde{X}_t) = \L(\tilde{X}'_t \vert \tilde{\tau}' > t)$ must coincide with the distribution $\mu_t = \L(X_t)$ from the solution constructed above. But the SDE \eqref{eq:fleming_viot}, where the flow of measures $\mu$ is viewed as a fixed input, exhibits uniqueness in marginal law (as can be seen from a simple inductive argument), so we can conclude that $\L(\tilde{X}_t) = \L(X_t)$. This completes the proof.
\end{proof}

\begin{remark}
Extending the work of Burdzy, Ho\l{}yst, Ingerman \& March \cite{burdzy_flem_viot_2000}, Tough \& Nolen \cite{tough_fleming_viot_2022} introduced a particle system associated with McKean--Vlasov SDE \eqref{eq:fleming_viot}, where, upon hitting $\partial D$, a particle is reinserted at a position sampled uniformly at random amongst the current locations of the remaining particles. They show (cf.\@ \cite[Theorem 2.9]{tough_fleming_viot_2022}) that, under suitable continuity assumptions on the drift coefficient, the empirical measure $\mu^N_t$ of the particle system of size $N \geq 1$ converges in probability to the conditional law $\mu_t = \L(X'_t \vert \tau' > t)$, where $X'$ solves the conditional McKean--Vlasov SDE \eqref{eq:conditional_mv}. The converge is with respect to the weak convergence of measures on the space $\P(D)$ and locally uniform in time. In view of Theorem \ref{thm:correspondence}, $\mu^N_t$ also converges to the time marginal of the Fleming--Viot dynamics of McKean--Vlasov type \eqref{eq:fleming_viot} and it stands to reason that this convergence can be extended to whole trajectories of the particles. This is an interesting avenue for future research.

Let us further note that \cite[Theorem 2.9]{tough_fleming_viot_2022} shows that the average number of reinsertions $F^N_t$ in the particle system up to time $t$ tends to $-\log \pr(\tau' > t)$. This is in accordance with our finding that $F_t = -\log \pr(\tau' > t)$ and implies that $F^N_t \to F_t$ in probability locally uniformly in time.
\end{remark}

We can now introduce a control problem for McKean--Vlasov SDE \eqref{eq:fleming_viot} in closed-loop form, which provides an alternative interpretation of the control problem for the conditional McKean--Vlasov SDE \eqref{eq:conditional_mv}. The controller's objective is to maximise the reward functional
\begin{equation}
    J_{\textup{FV}}(a) = \ev\biggl[\int_0^T f\bigl(t, X_t, \mu_t, a(t, X_t)\bigr) \, \d t - c F_T + g(\mu_T)\biggr]
\end{equation}
subject to McKean--Vlasov SDE \eqref{eq:fleming_viot} over all measurable functions $a \define [0, \infty) \times \R^d \to A$. Here $F_t = \sum_{k = 1}^{\infty} \pr(\tau_k \leq t)$ and $c \geq 0$ captures the cost of reinsertion. If $c = 0$, it follows from Theorem \ref{thm:correspondence} that $J_{\textup{FV}}(a) = J_{\textup{CP}}(a)$, where $J_{\textup{CP}}$ is the closed-loop reward functional for the control of conditional processes defined above Proposition \ref{prop:exist_unique}. Hence, the control problem for McKean--Vlasov SDE \eqref{eq:fleming_viot} and for the conditional McKean--Vlasov SDE \eqref{eq:conditional_mv} are equivalent. Even if $c > 0$, $J_{\textup{FV}}(a)$ can be written in terms of the conditional dynamics \eqref{eq:conditional_mv} by using the correspondence $F_t = -\log \pr(\tau' > t)$ established in the proof of Theorem \ref{thm:correspondence}. We leave the analysis of this McKean--Vlasov control problem for future work.

\appendix

\section{Hitting Times of Diffusions in Open Sets}

In this appendix, we derive regularity results for the hitting time of nondegenerate diffusions on the boundary of open sets. For the remainder of this section, let $\xi$ be an $\R^d$-valued $\F_0$-measurable random variable, $W$ be a $d$-dimensional Brownian motion, and $\sigma \in \R^{d \times d}$ be an invertible matrix. For an $\bb{F}$-progressively measurable process $\beta = (\beta_t)_{t \geq 0} \in L^{\infty}([0, \infty) \times \Omega; \R^d)$, we define $X^{\beta} = (X^{\beta}_t)_{t \geq 0}$ by $X^{\beta}_t = \xi + \int_0^t \beta_s \, \d s + \sigma W_t$ for $t \geq 0$.

\begin{lemma} \label{lem:bounds}
Let $D \subset \R^d$ be an open set, assume that $\xi$ takes values in $D$, and set $\tau_{\beta} = \inf\{t > 0 \define X^{\beta}_t \notin D\}$. Then for all $C > 0$ and $t \geq 0$, $\pr(\tau_{\beta} > t)$ is bounded away from zero uniformly over all $\bb{F}$-progressively measurable processes $\beta \in L^{\infty}([0, \infty) \times \Omega; \R^d)$ with $\lvert \beta_s \rvert < C$ for $\Leb \otimes \pr$-a.e.\@ $(s, \omega) \in [0, \infty) \times \Omega$.
\end{lemma}

\begin{proof}
Let us fix $C > 0$ and $t \geq 0$ and a sequence $(\beta^n)_{n \geq 1}$ in the space $L^{\infty}_C$ of $\bb{F}$-progressively measurable processes $\beta \in L^{\infty}([0, \infty) \times \Omega; \R^d)$ with $\lvert \beta_s \rvert < C$ for $\Leb \otimes \pr$-a.e.\@ $(s, \omega) \in [0, \infty) \times \Omega$ such that $\lim_{n \to \infty} \pr(\tau_n > t) = \inf_{\beta \in L^{\infty}_C} \pr(\tau_{\beta} > t)$, where $\tau_n = \tau_{\beta^n}$ and $X^n = X^{\beta^n}$. Since $(\beta^n)_{n \geq 1}$ is uniformly bounded, the Arzel\`a--Ascoli theorem guarantees the tightness of the sequence $(A^n)_{n \geq 1}$ with $A^n = \int_0^{\cdot} \beta^n_s \, \d s$. Hence, selecting a subsequence and enlarging the filtered probability space $(\Omega, \F, \bb{F}, \pr)$ if necessary, we may assume there exists a random variable $A$ with values in $C([0, \infty); \R^d)$ such that $(\xi, A^n, W)_{n \geq 1}$ converges weakly on $\R^d \times C([0, \infty); \R^d) \times C([0, \infty); \R^d)$ to $(\xi, A, W)$. In particular, it holds that $X^n \Rightarrow X = \xi + A + W$ on $C([0, \infty); \R^d)$. We set $\tau = \inf\{t > 0 \define X_t \notin D\}$. Now, we want to establish two facts: firstly, $\pr(\tau > t) > 0$ and, secondly, $\liminf_{n \to \infty} \pr(\tau_n > t) \geq \pr(\tau > t)$. The first statement can be easily proven by a change of measure argument through which we remove the drift of $X$. This relies on the uniform boundedness of $(\beta^n)_{n \geq 1}$, whereby $\lvert A_s - A_u\rvert \leq C \lvert s - u\rvert$, so that $A_s = \int_0^s \beta_u \, \d u$ for some $\beta \in L^{\infty}_C$. It also uses the invertibility of $\sigma$. We will skip the details here. For the second fact, let us note that $\tau_n \leq t$ if and only if $\inf_{s \in [0, t]} d(X^n_s, \partial D) = 0$. But the expression $C([0, \infty); \R^d) \ni x \mapsto \inf_{s \in [0, t]} d(x_s, \partial D)$ is continuous, meaning that the set of $x \in C([0, \infty); \R^d)$ for which $\inf_{s \in [0, t]} d(x_s, \partial D) = 0$ is closed. Consequently, the Portmanteau theorem implies
\begin{align*}
    \limsup_{n \to \infty} \pr(\tau_n \leq t) &= \limsup_{n \to \infty}\pr\Bigl(\inf_{s \in [0, t]} d(X^n_s, \partial D) = 0\Bigr) \\
    &\leq \pr\Bigl(\inf_{s \in [0, t]} d(X_s, \partial D) = 0\Bigr) \\
    &= \pr(\tau \leq t).
\end{align*}
From this we deduce $\liminf_{n \to \infty} \pr(\tau_n > t) \geq \pr(\tau > t) > 0$.
\end{proof}

From now on, we fix an $\bb{F}$-progressively measurable process $\beta = (\beta_t)_{t \geq 0} \in L^{\infty}([0, \infty) \times \Omega; \R^d)$ and set $X = X^{\beta}$.

\begin{proposition} \label{prop:hitting_no_atoms}
Let $D \subset \R^d$ be an open set, assume $\xi$ takes values in $D$, and set $\tau = \inf\{t > 0 \define X_t \notin D\}$. Then $\pr(\tau = t) = 0$ for all $t \geq 0$.
\end{proposition}

\begin{proof}
We can assume without loss of generality that $D$ includes the origin, that $\xi = 0 \in D$, that $\beta$ vanishes, and that $\sigma$ is the identity matrix in $d$ dimension. For more details, see the proof of Proposition \ref{prop:hitting_boundary} below.

For the purpose of deriving a contradiction, let us suppose there exists $t_0 > 0$ such that $\pr(\tau = t_0) > 0$. We will show that this implies $\pr(\tau = t) > 0$ for any $t \in (0, t_0]$, which is clearly a contradiction. For $t > 0$, we define $\pr_t = \pr(\cdot \vert \tau > t)$, let $\ev_t$ denote the expectation under $\pr_t$, and let $p_t$ be the density of the law of $W_t$ under $\pr_t$. This density exists because the law of $W_t$ under $\pr_t$ is absolutely continuous with respect to the Lebesgue measure on $\R^d$ and it holds that $p_t(x) > 0$ for a.e.\@ $x$ in the same connected component $C$ of $D$ as $0$. To verify the latter, it is enough to show that for any $z \in C$ and any $\delta > 0$ such that $B_{\delta}(z)$ lies in $C$, we have that $p_t(x) > 0$ for a.e.\@ $x \in B_{\delta}(z)$. Fix such $z \in C$ and $\delta > 0$ and choose a continuous path $\gamma \define [0, 1] \to C$ with $\gamma(0) = 0$ and $\gamma(1) = z$. For each $s \in [0, 1]$, we can find $\delta_s > 0$ with $B_{\delta_s}(\gamma(s)) \subset C$. Since the image of $\gamma$ is compact, we can select finitely many points $0 \leq t_1 < \dots < t_n \leq 1$ so that $\gamma([0, 1])$ is covered by $\bigcup_{i = 1}^n B_{\delta_{t_i}}(\gamma(t_i))$. Let us define
\begin{equation*}
    \tilde{C} = B_{\delta}(z) \cup \biggl(\bigcup_{i = 1}^n B_{\delta_{t_i}}(\gamma(t_i))\biggr),
\end{equation*}
which is an open connected subset of $C$ with Lipschitz boundary, containing the points $0$ and $z$. Set $\tilde{\tau} = \inf\{s > 0 \define W_s \notin \tilde{C}\}$ and let $\tilde{p}_t$ be the density of the measure $\pr(W_t \in \cdot,\, \tilde{\tau} > t)$, which clearly satisfies $\tilde{p}_t(x) > 0$ for a.e.\@ $x \in \tilde{C}$. Since $\tilde{\tau} \leq \tau$, it holds that
\begin{equation*}
    p_t(x) \geq \frac{\tilde{p}_t(x)}{\pr(\tau > t)} > 0
\end{equation*}
for a.e.\@ $x \in B_{\delta}(z) \subset \tilde{C}$, as required.

Next, let us fix $t \in (0, t_0]$, set $\epsilon = t/2$, and define $q = \pr_{t_0 - \epsilon}(\tau = t_0) > 0$. Then, we let $K$ be the set of $x \in C$ such that $\pr_{t_0 - \epsilon}(\tau = t_0 \vert W_{t_0 - \epsilon} = x) \geq q/2$. Note that the set $K$ is measurable and has positive Lebesgue measure, for otherwise it would hold that
\begin{align*}
    \pr_{t_0 - \epsilon}(\tau = t_0) &= \int_C \pr_{t_0 - \epsilon}(\tau = t_0 \vert W_{t_0 - \epsilon} = x) p_{t_0 - \epsilon}(x) \, \d x \\
    &\leq \int_K p_{t_0 - \epsilon}(x) \, \d x + \int_{K^c} \frac{q}{2} p_{t_0 - \epsilon}(x) \, \d x \\
    &\leq \frac{q}{2},
\end{align*}
which cannot be true since $q = \pr_{t_0 - \epsilon}(\tau = t_0) > 0$ by assumption.

Now, by the Markov property of Brownian motion, it holds that $W^{(s)}$, defined by $W^{(s)}_u = W_{s + u} - W_s$, is a Brownian motion independent of $\F_s$. Hence, setting $\rho_s(x) = \inf\{u > s \define W^{(s)}_{u - s} + x \notin D\}$ for $x \in C$, we find
\begin{align} \label{eq:ind_rewrite}
    \pr_{t_0 - \epsilon}(\tau = t_0 \vert W_{t_0 - \epsilon}) &= \frac{\pr\bigl(\tau = t_0,\, \tau > t_0 - \epsilon \big\vert W_{t_0 - \epsilon}\bigr)}{\pr(\tau > t_0 - \epsilon \vert W_{t_0 - \epsilon})} \notag \\
    &= \frac{\pr\bigl(\rho_{t_0 - \epsilon}(W_{t_0 - \epsilon}) = t_0,\, \tau > t_0 - \epsilon \big\vert W_{t_0 - \epsilon}\bigr)}{\pr(\tau > t_0 - \epsilon \vert W_{t_0 - \epsilon})} \notag \\
    &= \pr\bigl(\rho_{t_0 - \epsilon}(W_{t_0 - \epsilon}) = t_0 \big\vert W_{t_0 - \epsilon}\bigr).
\end{align}
This implies that $\pr_{t_0 - \epsilon}(\tau = t_0 \vert W_{t_0 - \epsilon} = x) = \pr\bigl(\rho_{t_0 - \epsilon}(W_{t_0 - \epsilon}) = t_0 \big\vert W_{t_0 - \epsilon} = x\bigr)$ for a.e.\@ $x \in C$. However, it clearly holds that
\begin{equation*}
    \pr\bigl(\rho_{t_0 - \epsilon}(W_{t_0 - \epsilon}) = t_0 \big\vert W_{t_0 - \epsilon} = x\bigr) = \pr\bigl(\rho_{t - \epsilon}(W_{t - \epsilon}) = t \big\vert W_{t - \epsilon} = x\bigr)
\end{equation*}
and, similarly to Equation \eqref{eq:ind_rewrite}, we have $\pr\bigl(\rho_{t - \epsilon}(W_{t - \epsilon}) = t \big\vert W_{t - \epsilon} = x\bigr) = \pr_{t - \epsilon}(\tau = t \vert W_{t - \epsilon} = x)$ for a.e.\@ $x \in C$. Thus, in summary, we have deduced that
\begin{equation*}
    \pr_{t_0 - \epsilon}(\tau = t_0 \vert W_{t_0 - \epsilon} = x) = \pr_{t - \epsilon}(\tau = t \vert W_{t - \epsilon} = x)
\end{equation*}
for a.e.\@ $x \in C$, which in turn implies that $\pr_{t - \epsilon}(\tau = t \vert W_{t - \epsilon} = x) \geq q/2$ for a.e.\@ $x \in K$. Then, we compute
\begin{align*}
    \pr(\tau = t) &= \pr(\tau > t - \epsilon) \pr_{t - \epsilon}(\tau = t) \\
    &\geq \pr(\tau > t - \epsilon) \int_K \pr_{t - \epsilon}(\tau = t \vert W_{t - \epsilon} = x) p_{t - \epsilon}(x) \, \d x \\
    &\geq \frac{q \pr(\tau > t - \epsilon)}{2} \int_K p_{t - \epsilon}(x) \, \d x.
\end{align*}
But by our earlier remark, $p_{t - \epsilon}(x) > 0$ for a.e.\@ $x \in C$, so the right-hand side is positive, as required. This concludes the proof.
\end{proof}

\begin{proposition} \label{prop:hitting_boundary}
Let $D \subset \R^d$ be an open set that satisfies the Poincar\'e--Zaremba cone condition and assume that $\xi$ takes values in $\partial D$. Then a.s.\@ for all $\epsilon > 0$ there exists $t \in (0, \epsilon]$ such that $X_t \notin \bar{D}$.
\end{proposition}

\begin{proof}
First, we reduce to the case that $\beta$ vanishes and $\sigma = I_d$. Using that $\beta$ is uniformly bounded and $\sigma$ is invertible, we can apply the change of measure 
\begin{equation*}
    \frac{\d \qr}{\d \pr}\bigg\vert_{\F_1} = \cal{E}\biggl(\int_0^{\cdot} - \sigma^{-1} \beta_t \, \d W_t\biggr)_1,
\end{equation*}
which is well-defined because the stochastic exponential is a martingale by Novikov's condition. Under $\qr$, the process $B = (B_t)_{t \geq 0}$ defined by $B_t = \int_0^{t \land 1} \beta_s \, \d s + W_t$ is a Brownian motion. Now, since we are interested in proving that a statement has probability one and $\qr \sim \pr$, we may show the result under $\qr$, where $X$ is driftless. Thus, we may as well assume that $\beta$ vanishes. Next, since $\sigma$ is invertible, we may consider the process $X' = (X'_t)_{t \geq 0}$ defined by $X'_t = \sigma^{-1}X_t = \sigma^{-1} \xi + W_t$. Let us set $D' = \sigma^{-1}D = \{\sigma^{-1} x \define x \in D\}$, so that $D'$ satisfies the Poincar\'e--Zaremba cone condition, $\partial D' = \{\sigma^{-1} x \define x \in \partial D\}$, and $\bar{D}' = \{\sigma^{-1} x \define x \in \bar{D}\}$. The process $X'$ is started from the $\partial D'$-valued random variable $\sigma^{-1} \xi$ and a.s.\@ for all $\epsilon \in (0, 1]$ there exists $t \in (0, \epsilon]$ such that $X_t \notin \bar{D}$ if and only if a.s.\@ for all $\epsilon \in (0, 1]$ there exists $t \in (0, \epsilon]$ such that $X'_t \notin \bar{D}'$. Consequently, it suffices to establish the proposition for $\sigma = I_d$. 

We want to further simplify to the case that $\xi = x_0$ for some $x_0 \in \partial D$. For any $A \in \F$, we can write $\pr(A) = \int_{\partial D} \pr(A \vert \xi = x) \, \d \L(\xi)(x)$. Hence, if we can prove the proposition under the law $\pr_x = \pr(\cdot \vert \xi = x)$ for $\L(\xi)$-a.e.\@ $x \in \partial D$, then the general case follows as well. However, for $\L(\xi)$-a.e.\@ $x \in \partial D$, $W$ is a Brownian motion under $\pr(\cdot \vert \xi = x)$ and $X$ satisfies $X_t = x + W_t$, meaning that the assumptions of the proposition hold with $\xi = x$. Consequently, we are done if we can prove the proposition for the case that $\xi = x$ for some $x \in \partial D$. We shall do this next.

Owing to the Poincar\'e--Zaremba cone condition, we can find an open cone $C$ with apex $x$ and a $\delta > 0$ such that $C \cap B_{\delta}(x) \subset \bar{D}^c$. We must show that a.s.\@ for all $\epsilon \in (0, 1]$ there exists $t \in (0, \epsilon]$ such that $X_t \notin \bar{D}$. Note that this event, call it $E$, belongs to $\F^W_{0+}$, where $(\F^W_t)_{t \geq 0}$ denotes the filtration generated by $W$ and $(\F^W_{t+})_{t \geq 0}$ is its right-continuous extension. Hence, Blumenthal's zero-one law for Brownian motion implies that $E$ occurs either with probability zero or one. Consequently, it is enough to show that $\pr(E) > 0$. However, by continuity from above of the measure $\pr$, we have
\begin{equation} \label{eq:cont_measure}
    \pr(E) = \lim_{\epsilon \searrow 0} \pr\bigl(\exists t \in (0, \epsilon] \define X_t \in \bar{D}^c\bigr) \geq \lim_{\epsilon \searrow 0} \pr\bigl(\exists t \in (0, \epsilon] \define X_t \in C \cap B_{\delta}(x)\bigr).
\end{equation}
Now, the scaling properties of Brownian motion imply that the process $W^{\epsilon} = (W^{\epsilon}_t)_{t \geq 0}$ defined by $W^{\epsilon}_t = \epsilon^{-1/2} W_{\epsilon t}$ is a Brownian motion. Consequently,
\begin{align*}
    \pr\bigl(\exists t \in (0, \epsilon] \define X_t \in C \cap B_{\delta}(x)\bigr) &= \pr\bigl(\exists t \in (0, 1] \define X_{\epsilon t} \in C \cap B_{\delta}(x)\bigr) \\
    &= \pr\bigl(\exists t \in (0, 1] \define x + \epsilon^{1/2} W^{\epsilon}_t \in C \cap B_{\delta}(x)\bigr) \\
    &= \pr\bigl(\exists t \in (0, 1] \define x +  W^{\epsilon}_t \in C \cap B_{\epsilon^{-1/2}\delta}(x)\bigr) \\
    &\geq \pr\bigl(\exists t \in (0, 1] \define x +  W_t \in C \cap B_{\delta}(x)\bigr),
\end{align*}
where we used in the third equality that $C$ is a cone with apex $x$. The probability on the right-hand side is clearly positive, so in view of \eqref{eq:cont_measure}, we have 
\begin{equation*}
    \pr(E) \geq \lim_{\epsilon \searrow 0} \pr\bigl(\exists t \in (0, \epsilon] \define X_t \in C \cap B_{\delta}(x)\bigr) \geq \pr\bigl(\exists t \in (0, 1] \define x +  W_t \in C \cap B_{\delta}(x)\bigr) > 0,
\end{equation*}
as required. This concludes the proof.
\end{proof}

\section{Notions of Uniqueness for McKean--Vlasov SDE \texorpdfstring{\eqref{eq:fleming_viot}}{(3.1)}} \label{app:yw}
 
We first establish a Yamada--Watanabe type theorem that connects the notions of pathwise uniqueness and uniqueness in law for McKean--Vlasov SDE \eqref{eq:fleming_viot} introduced in Definition \ref{def:uniqueness}. 

\begin{theorem} \label{thm:yw}
Pathwise uniqueness for McKean--Vlasov SDE \eqref{eq:fleming_viot} implies uniqueness in law.
\end{theorem}

\begin{proof}
We proceed in two steps. First, we will show that changing the representation of $\P(\R^d)$ underlying a weak solution of McKean--Vlasov SDE \eqref{eq:fleming_viot} does not change its law. Then, for any two weak solutions $X^1$ and $X^2$ with possibly different representations of $\P(\R^d)$ defined on possibly different probability spaces, we will construct solutions $\tilde{X}^1$ and $\tilde{X}^2$ with the same underlying representation of $\P(\R^d)$ defined on the same stochastic setup (cf.\@ paragraph above Definition \ref{def:weak_solution}), such that $\tilde{X}^i \sim X^i$ for $i = 1$, $2$. From this and the assumed pathwise uniqueness for McKean--Vlasov SDE \eqref{eq:fleming_viot}, we deduce
\begin{equation} \label{eq:equal_bridge}
    X^1 \sim \tilde{X}^1 \sim \tilde{X}^2 \sim X^2,
\end{equation}
implying uniqueness in law.

\textit{Step 1}: Fix a weak solution $X$ of McKean--Vlasov SDE \eqref{eq:fleming_viot} with underlying representation $H$ of $\P(\R^d)$ and suppose we wish to change this representation to another representation $\tilde{H}$. We will inductively construct another solution $\tilde{X}$ on the same stochastic setup with underlying representation $\tilde{H}$ and such that $\tilde{X} \sim X$. We begin by setting $\tilde{\tau}_1 = \tau_1$ and defining $(\tilde{X}_t)_{t \in [0, \tilde{\tau}_1)}$ by $\tilde{X}_t = X_t$ for $t \in [0, \tilde{\tau}_1)$. Now, suppose that $\tilde{X}$ is constructed up to the $k$th reinsertion time $\tilde{\tau}_k$ and $(\tilde{X}_t)_{t \in [0, \tilde{\tau}_k)} \sim (X_t)_{t \in [0, \tau_k)}$ for some $k \geq 1$, where $\tau_k$ is the $k$th time that $X$ is reinserted. Then setting $\tilde{\xi}_k = \tilde{H}(\mu_{\tilde{\tau}_k}, U_k)$, where $\mu_t = \L(X_t)$, it holds that
\begin{align} \label{eq:insertion_equidistr}
    \bigl((\tilde{X}_t)_{t \in [0, \tilde{\tau}_k)}, \tilde{\xi}_k\bigr) &= \bigl((\tilde{X}_t)_{t \in [0, \tilde{\tau}_k)}, \tilde{H}(\mu_{\tilde{\tau}_k}, U_k)\bigr) \notag \\
    &\sim \bigl((X_t)_{t \in [0, \tau_k)}, \tilde{H}(\mu_{\tau_k}, U_k)\bigr) \notag \\
    &\sim \bigl((X_t)_{t \in [0, \tau_k)}, H(\mu_{\tau_k}, U_k)\bigr) \notag \\
    &= \bigl((X_t)_{t \in [0, \tau_k)}, \xi_k\bigr).
\end{align}
The third line follows from the fact that
\begin{equation*}
    \L\bigl(\tilde{H}(\mu_{\tau_k}, U_k)\big\vert (X_t)_{t \in [0, \tau_k)}\bigr) = \mu_{\tau_k} = \L\bigl(H(\mu_{\tau_k}, U_k)\big\vert (X_t)_{t \in [0, \tau_k)}\bigr)
\end{equation*}
since $U_k$ is independent of $(X_t)_{t \in [0, \tau_k)}$ and both $H$ and $\tilde{H}$ are representations of $\P(\R^d)$. Now, setting $\tilde{X}_{\tilde{\tau}_k} = \tilde{\xi}_k$, we have that $(\tilde{X}_t)_{t \in [0, \tilde{\tau}_k]}$ solves McKean--Vlasov SDE \eqref{eq:fleming_viot} up to and including time $\tilde{\tau}_k$ and \eqref{eq:insertion_equidistr} implies that $(\tilde{X}_t)_{t \in [0, \tilde{\tau}_k]} \sim (X_t)_{t \in [0, \tau_k]}$. Next, we extend $\tilde{X}$ by solving SDE \eqref{eq:fleming_viot} between $\tilde{\tau}_k$ and $\tilde{\tau}_{k + 1} = \inf\{t > \tilde{\tau}_k \define \tilde{X}_{t-} \notin D\}$ with $\mu_t = \L(X_t)$, $t \geq 0$, viewed as a fixed input. This SDE has a unique strong solution, so we get $(\tilde{X}_t)_{t \in [0, \tilde{\tau}_{k + 1})} \sim (X_t)_{t \in [0, \tau_{k + 1})}$, which concludes the induction step. Thus, we obtain the desired weak solution $\tilde{X}$ of McKean--Vlasov SDE \eqref{eq:fleming_viot} with underlying representation $\tilde{H}$ and such that $\tilde{X} \sim X$.

\textit{Step 2}: Fix two weak solutions $X^1$ and $X^2$ on possibly different stochastic setups $(\Omega^i, \F^i, \bb{F}^i, \pr^i, W^i, (U^i_k)_{k \geq 1})$, $i = 1$, $2$. Owing to Step 1, we may assume that $X^1$ and $X^2$ share the same representation of $\P(\R^d)$. Now, we proceed as in the proof of the classical Yamada--Watanabe theorem \cite[Proposition 1]{yamada_uniqueness_1971} to transfer $X^1$ and $X^2$ to a common stochastic setup. Let us define the regular conditional distributions
\begin{equation*}
    P^i_{x, w, u} = \L\bigl(X^i \big\vert X^i_0 = x, W^i = w, (U^i_k)_{k \geq 1} = u\bigr)
\end{equation*}
on $D_{\R^d}[0, \infty)$ for $x \in \R^d$, $w \in C([0, \infty); \R^d)$, and $u \in [0, 1]^{\bb{N}}$. Here $D_{\R^d}[0, \infty)$ is the space of c\`adl\`ag paths $[0, \infty) \to \R^d$ endowed with the $J1$-topology. Then, we define the probability space $(\Omega, \F, \pr)$ as follows: $\Omega = D_{\R^d}[0, \infty) \times D_{\R^d}[0, \infty) \times \R^d \times C([0, \infty); \R^d) \times [0, 1]^{\bb{N}}$, $\F$ is the Borel $\sigma$-algebra on $\Omega$, and $\pr$ is determined by
\begin{equation*}
    \pr\bigl(A_1 \times A_2 \times A_3 \times A_4 \times A_5\bigr) = \int_{A_3 \times A_4 \times A_5} P^1_{x, w, u}(A_1) P^2_{x, w, u}(A_2) \, \d Q(x, w, u)
\end{equation*}
for Borel measurable sets $A_1 \subset D_{\R^d}[0, \infty)$, $A_2 \subset D_{\R^d}[0, \infty)$, $A_3 \subset \R^d$, $A_4 \subset C([0, \infty); \R^d)$, and $A_5 \subset [0, 1]^{\bb{N}}$ and $Q = \L(X^1_0, W^1, (U^1_k)_{k \geq 1})$. We let $(\tilde{X}^1, \tilde{X}^2, \xi, W, (U_k)_{k \geq 1})$ be the canonical random element on $\Omega$ and define the filtration $\bb{F} = (\F_t)_{t \geq 0}$ by
\begin{equation*}
    \F_t = \sigma\bigl(\tilde{X}^1_s, \tilde{X}^2_s, \xi, W_s, U_k \define s \in [0, t], k \geq 1\bigr). 
\end{equation*}
Arguing as in the proof of \cite[Proposition 1]{yamada_uniqueness_1971}, we can show that
\begin{equation*}
    (X^i, X^i_0, W^i, (U^i_k)_{k \geq 1}) \sim (\tilde{X}^i, \xi, W, (U_k)_{k \geq 1})
\end{equation*}
for $i = 1$, $2$. In particular, $\tilde{X}^1$ and $\tilde{X}^2$ are weak solutions to McKean--Vlasov SDE \eqref{eq:fleming_viot} on the same stochastic setup $(\Omega, \F, \bb{F}, \pr, W, (U_k)_{k \geq 1})$, with the same initial condition $\xi$ and the same underlying representation of $\P(\R^d)$. By pathwise uniqueness for McKean--Vlasov SDE \eqref{eq:fleming_viot}, we have $\tilde{X}^1 \sim \tilde{X}^2$, which gives \eqref{eq:equal_bridge}.
\end{proof}

Next, we provide a counterexample that shows that uniqueness in law does not hold if we do not insist that the sequence $(\xi_k)_{k \geq 1}$ in Definition \ref{def:weak_solution} of weak solutions is constructed via a representation of $\P(\R^d)$ and an independent sequence of uniform random variables. To be clear, by such a weak solution without underlying representation of $\P(\R^d)$, we simply mean a stochastic setup $(\Omega, \F, \bb{F}, \pr, W, (U_k)_{k \geq 1})$ carrying an $\R^d$-valued c\`adl\`ag $\bb{F}$-adapted stochastic process $X$ such that SDE \eqref{eq:fleming_viot} holds and $\mu_t = \L(X_t)$ for all $t \geq 0$.

Note that while Example \ref{ex:counter} below shows that the law of two such weak solutions without underlying representation of $\P(\R^d)$ may be distinct, their marginal laws still coincide. Indeed, a careful inspection of the proof of Theorem \ref{thm:correspondence} reveals that only the properties of weak solutions in the above sense, without underlying representation of $\P(\R^d)$, are used to establish the correspondence between the solution's marginal law and the conditional law $\L(X'_t \vert \tau' > t)$ of the unique strong solution $X'$ of the conditional McKean--Vlasov SDE \ref{eq:conditional_mv_closed}, where $\tau' = \inf\{t > 0 \define X'_t \notin D\}$. Thus, uniqueness in marginal law also holds for weak solutions without underlying representation of $\P(\R^d)$.

\begin{example} \label{ex:counter}
Consider the domain $[-1, 1] \times [-2, 2]$ in $\R^2$, initial condition $\xi$ distributed according to $\frac{1}{2}(\delta_{-z} + \delta_{z})$ with $z = (0, 1)^{\top} \in D$, vanishing drift $b(t, x, m, a) = 0$, and unit diffusion coefficient $\sigma = I_2$, where $I_2$ denotes the identity matrix in two dimensions. Let us set $A_{\pm} = \{\xi = \pm z\}$ and then define $\tilde{W} = (\tilde{W}^1, \tilde{W}^2)^{\top}$ by $\tilde{W}^1 = W^1$ and $\tilde{W}^2 = - \mathbf{1}_{A_-} W^2 + \mathbf{1}_{A_+} W^2$, where $W = (W^1, W^2)^{\top}$. Note that $\tilde{W}$ is an $\bb{F}$-Brownian motion. Then we let $X$ be the strong solution to McKean--Vlasov SDE \eqref{eq:fleming_viot} from Theorem \ref{thm:correspondence} for some given representation $H$ of $\P(\R^d)$ on the stochastic setup $(\Omega, \F, \bb{F}, \pr, \tilde{W}, (U_k)_{k \geq 1})$ for an i.i.d\@ sequence $(U_k)_{k \geq 1}$ of random variables uniformly distributed on $[0, 1]$. We will construct an alternative solution $\tilde{X}$ with a different law. This construction will be similar to the one at the beginning of the proof of Theorem \ref{thm:correspondence}. However, we will suitably modify the definition of $\xi_1$. To that end, let $\tau \define D \times C([0, \infty); \R^2) \to [0, \infty)$ be defined by $\tau(x, w) = \inf\{t > 0 \define x + w_t \notin D\}$, so that $\tau_1 = \tau(\xi, \tilde{W})$.

Now, suppose $\tilde{X}$ is obtained by solving SDE \eqref{eq:fleming_viot}, using the same Brownian motion $\tilde{W}$, up to the first hitting time $\tilde{\tau}_1 = \inf\{t > 0 \define \tilde{X}_{t-} \notin D\} = \tau(\xi, \tilde{W})$, so that $X_{\cdot \land \tau_1} = \tilde{X}_{\cdot \land \tilde{\tau}_1}$. We claim that $\tilde{\tau}_1 = \tau(z, W)$, so that $\tilde{\tau}_1$ is independent of $\xi$. Indeed, on $A_+$, the equality $\tilde{\tau}_1 = \tau(z, W)$ follows from the fact that $\tilde{W} = W$. On $A_-$, we have
\begin{align*}
    \tilde{\tau}_1 &= \inf\{t > 0 \define -z + \tilde{W}_t \notin D\} \\
    &= \inf\bigl\{t > 0 \define (0 + W^1_t, -1 - W^2_t)^{\top} \notin D\bigr\} \\
    &= \inf\{t > 0 \define z + W_t \notin D\} \\
    &= \tau(z, W),
\end{align*}
where we used in the third equality that the domain is symmetric under reflections across the $x$-axis. We will now define $\tilde{\xi}_1 = (\tilde{\xi}_1^1, \tilde{\xi}_1^2)^{\top}$ by $\tilde{\xi}_1^1 = H^1(\mu_{\tilde{\tau}_1}, U_1)$ and
\begin{equation*}
    \tilde{\xi}_1^2 = \sgn(\xi^2) \lvert H^2(\mu_{\tilde{\tau}_1}, U_1) \rvert = \sgn(\xi^2) \lvert H^2(\mu_{\tau(z, W)}, U_1) \rvert,
\end{equation*}
where $\xi^i$ and $H^i$ denote the $i$th component of $\xi$ and $H$, respectively, and $\sgn(x) = \mathbf{1}_{x > 0} - \mathbf{1}_{x \leq 0}$ for $x \in \R$. We claim that $\L(\tilde{\xi}_1 \vert \tilde{\tau}_1 = t) = \mu_t$ for $t \geq 0$. To prove this, we establish the equivalent result $\L(\tilde{\xi}_1, \tilde{\tau}_1) = \L(\xi_1, \tau_1)$. For $F \define \R^2 \times [0, \infty) \to \R$ measurable and bounded, we have
\begin{equation} \label{eq:law_decomp}
    \ev F(\tilde{\xi}_1, \tilde{\tau}_1) = \ev\bigl[\mathbf{1}_{A_-}F(\tilde{\xi}_1, \tilde{\tau}_1)\bigr] + \ev\bigl[\mathbf{1}_{A_+}F(\tilde{\xi}_1, \tilde{\tau}_1)\bigr].
\end{equation}
Now, by the definition of $\tilde{\xi}_1$ and since $\xi$ is independent of $(W, U_1)$, we find
\begin{align*}
    \ev\bigl[\mathbf{1}_{A_{\pm}} F(\tilde{\xi}_1, \tilde{\tau}_1)\bigr] &= \ev\Bigl[\mathbf{1}_{A_{\pm}} F\Bigl(H^1(\mu_{\tau(z, W)}, U_1), \pm\bigl\lvert H^2(\mu_{\tau(z, W)}, U_1) \bigr\rvert, \tau(z, W)\Bigr)\Bigr] \\
    &= \frac{1}{2}\ev\Bigl[F\Bigl(H^1(\mu_{\tau(z, W)}, U_1), \pm\bigl\lvert H^2(\mu_{\tau(z, W)}, U_1) \bigr\rvert, \tau(z, W)\Bigr)\Bigr] \\
    &= \frac{1}{2} \ev\bigl[F\bigl(\xi_1^1, \pm \lvert \xi_1^2\rvert, \tau_1\bigr)\bigr].
\end{align*}
The last equality follows from the fact that $\tau(z, W) = \tilde{\tau}_1 = \tau_1$ and the construction of $\xi_1$ via the representation $H$. Inserting this back into \eqref{eq:law_decomp} yields
\begin{align*}
    \ev F(\tilde{\xi}_1, \tilde{\tau}_1) &= \frac{1}{2}\Bigl(\ev\bigl[F\bigl(\xi_1^1, - \lvert \xi_1^2\rvert, \tau_1\bigr)\bigr] + \ev\bigl[F\bigl(\xi_1^1, \lvert \xi_1^2\rvert, \tau_1\bigr)\bigr]\Bigr) \\
    &= \frac{1}{2}\Bigl(\ev\bigl[F\bigl(\xi_1^1, -\xi_1^2, \tau_1\bigr)\bigr] + \ev\bigl[F\bigl(\xi_1^1, \xi_1^2, \tau_1\bigr)\bigr]\Bigr).
\end{align*}
Finally, we claim that $\L((\xi_1^1, -\xi_1^2)^{\top}, \tau_1) = \L(\xi_1, \tau_1)$, which in view of the above gives the desired equality $\ev F(\tilde{\xi}_1, \tilde{\tau}_1) = \ev F(\xi_1, \tau_1)$. To prove the claim, we simply note that due to the reflection symmetry of $D$ across the $x$-axis, the process $(X^1, -X^2)^{\top}$, where $X^1$ and $X^2$ are the components of $X$, is a solution to McKean--Vlasov SDE \eqref{eq:fleming_viot} with initial condition $(\xi^1, -\xi^2)^{\top}$, Brownian motion $(\tilde{W}^1, -\tilde{W}^2)^{\top}$, and representation $(H^1, -H^2)$ of $\P(\R^d)$. Note further that the first hitting time of $(X^1, -X^2)^{\top}$ on the boundary of $D$ coincides with $\tau_1$. Consequently, we have $((X^1, -X^2)^{\top}, \tau_1) \sim (X, \tau_1)$ by uniqueness in law of McKean--Vlasov SDE \eqref{eq:fleming_viot}. This proves the claim and establishes the identity $\L(\tilde{\xi}_1 \vert \tilde{\tau}_1 = t) = \mu_t$. In particular, by completing the construction of $\tilde{X}$, continuing as in the proof of Theorem \ref{thm:correspondence}, we obtain a solution of McKean--Vlasov SDE \eqref{eq:fleming_viot}. However, while $\xi_1$ is independent of $\xi$ due to the independence of $\tau_1 = \tau(z, W)$ from $\xi$, the random vectors $\tilde{\xi}_1^2$ and $\xi^2$ are clearly correlated since their signs coincide by definition. Hence, the law of $X$ and $\tilde{X}$ are distinct.
\end{example}

\section*{Acknowledgement}
The author is thankful to Andreas S\o{}jmark for insightful discussions and valuable suggestions and to Luciano Campi for useful comments on the manuscript. The author thanks the anonymous referee for great suggestions for improvement.

\sloppypar
\printbibliography

\end{document}